\documentclass[11pt]{article}

\usepackage{amsfonts,amsmath,amsthm,amssymb}
\usepackage{fullpage}
\usepackage{graphics,epsfig}
\usepackage{comment}
\newtheorem{theorem}{Theorem}[section]

\newtheorem{corollary}[theorem]{Corollary}
\newtheorem{lemma}[theorem]{Lemma}

\newtheorem{conjecture}{Conjecture}[section]
\def\cQ{\mathcal{Q}}
\def\cQg{\mathcal{Q}_g}
\def\Qg{Q_g}
\def\cR{\mathcal{R}}
\def\cL{\mathcal{L}}
\def\cRg{\mathcal{R}_g}
\def\Rg{R_g}
\def\cS{\mathcal{S}}
\def\cH{\mathcal{H}}
\def\cSg{\mathcal{S}_g}
\def\Sg{S_g}
\def\Sgei{S_g^{\mathrm{ew}=2i}}
\def\Sgci{S_g^{C=2i}}
\def\cSgei{\cS_g^{\mathrm{ew}=2i}}
\def\cSgci{\cS_g^{C=2i}}
\def\cE{\mathcal{E}}
\def\cG{\mathcal{G}}

\def\qedclaim{\hfill$\triangle$\smallskip}

\def\cLhc{\cL^{\mathrm{con}}}
\def\cLhs{\cL^{\mathrm{spl}}}
\def\bX{\overline{X}}
\def\bL{\overline{L}}

\def\bh{\overline{h}}
\def\cC{\mathcal{C}}
\def\cB{\mathcal{B}}
\def\cT{\mathcal{T}}
\def\fw{\mathrm{fw}}
\def\tcSof{\cS^{\Diamond}}
\def\tcSot{\cS^{\rVert}}
\def\tSof{S^{\Diamond}}
\def\tSot{S^{\rVert}}
\def\tcRof{\cR^{\Diamond}}
\def\tRof{R^{\Diamond}}
\def\tcQof{\cQ^{\Diamond}}
\def\tQof{Q^{\Diamond}}
\def\tA{\widetilde{A}}
\def\tPg{\widetilde{P}_g}

\def\tg{\tilde{g}}
\def\th{\widetilde{h}}
\def\cF{\mathcal{F}}

\def\Qgr{\overrightarrow{Q}_g}
\def\Mgr{\overrightarrow{M}_g}
\def\tDelta{\widetilde{\Delta}}
\def\fw{\mathrm{fw}}
\def\ew{\mathrm{ew}}
\newcommand{\cacher}[1]{}
\def\cZ{\mathcal{Z}}

\title{Asymptotic enumeration and limit laws for graphs of fixed genus}

\author{Guillaume Chapuy\thanks{Supported by a CNRS/PIMS postdoctoral fellowship. 
   Partial support from the European Grant ERC StG 208471 -- ExploreMaps}\\
\small   {Department of Mathematics}\\
\small   {Simon Fraser University}\\
\small   {Burnaby, B.C. V5A 1S6, Canada} \\
\small    {\tt gchapuy@sfu.ca}  \\
\and \'Eric Fusy\\
\small {Laboratoire d'Informatique}\\
\small {\'Ecole Polytechnique}\\  
\small {91128 Palaiseau Cedex, France}\\
\small    {\tt fusy@lix.polytechnique.fr}  
\and Omer Gim\'{e}nez\\
\small {Dept. de Llenguatges i Sistemes Inform\`atics}\\
\small {Universitat Polit\`ecnica de Catalunya}\\
\small {Jordi Girona, 1-3}\\
\small {08034 Barcelona, SPAIN}\\
\small    {\tt omer.gimenez@gmail.com}  
\and Bojan Mohar\thanks{Supported in part by the Research Grant
   P1--0297 (Slovenia), by an NSERC Discovery Grant (Canada)
   and by the Canada Research Chair program.}~\thanks{On leave from:
   IMFM \& FMF, Department of Mathematics, University of Ljubljana,
   Ljubljana, Slovenia.}\\
\small   {Department of Mathematics}\\
\small   {Simon Fraser University}\\
\small   {Burnaby, B.C. V5A 1S6, Canada} \\
\small    {\tt mohar@sfu.ca}
\and Marc Noy\\
\small {Departament de Matem\`atica Aplicada II}\\
\small {Universitat Polit\`{e}cnica de Catalunya}\\  
\small {Jordi Girona 1-3}\\
\small {08034 Barcelona, Spain} \\
\small    {\tt marc.noy@upc.edu}  
}

\begin{document}

\maketitle

\begin{abstract}
It is shown that the number of labelled graphs with $n$ vertices that
can be embedded in the orientable surface $\mathbb{S}_g$ of genus
$g$ grows asymptotically like
$$
c^{(g)}n^{5(g-1)/2-1}\gamma^n n!
$$
where $c^{(g)} >0$, and $\gamma \approx 27.23$ is the exponential
growth rate of planar graphs. This generalizes the result for the
planar case  $g=0$, obtained by Gim\'{e}nez and Noy.

An analogous result for non-orientable surfaces is obtained. In
addition, it is proved that several parameters of interest behave
asymptotically as in the planar case. It follows, in particular, that
a random graph embeddable in $\mathbb{S}_g$ has a unique
2-connected component of linear size with high probability.
\end{abstract}


\section{Introduction and statement of main results}

It has been shown by Gim\'{e}nez and Noy \cite{gimeneznoy} that the
number of planar graphs with $n$ labelled vertices grows
asymptotically as
$$
    c \cdot n^{-7/2} \gamma^n n!
$$
where $c>0$ and $\gamma \approx 27.23$ are well defined analytic
constants. Since planar graphs are precisely those that can be
embedded in the sphere, it is natural to ask about the number of
graphs that can be embedded in a given surface.

In what follows, graphs are simple and  labelled with $V =
\{1,2,\dots,n\}$, so that isomorphic graphs are considered
different unless they have exactly the same edges. Let
$\mathbb{S}_g$ be the orientable surface of genus $g$, that is, a
sphere with $g$ handles, and let $a_n^{(g)}$ be the number of
graphs with $n$ vertices embeddable in $\mathbb{S}_g$.
A first approximation to the magnitude of these numbers was given
by McDiarmid \cite{colin}, who showed that
$$
   \lim_{n\to\infty} \Bigg({a_n^{(g)} \over n!}\Bigg)^{1/n} =
   \gamma.
   $$
This establishes the \emph{exponential growth} of the $a_n^{(g)}$,
which is the same as for planar graphs and does not depend on the
genus.

In this paper we provide a considerable refinement and obtain a
sharp estimate, showing how the genus comes into play in the
\emph{subexponential growth}. In the next statements, $\gamma$ is
the exponential growth rate of planar graphs, and $\gamma_{\mu}$
is the exponential growth rate of planar graphs with $n$ vertices
and $\lfloor \mu n\rfloor$ edges. Both $\gamma$ and the function
$\gamma_\mu$ are determined analytically in~\cite{gimeneznoy}.


\begin{theorem}\label{theo:main}
For $g\geq 0$, the number $a_n^{(g)}$ of graphs with $n$ vertices
that can be embedded in the orientable surface  $\mathbb{S}_g$ of
genus $g$ satisfies
\begin{equation}\label{eq:asympt-orient}
a_n^{(g)} \sim \,c^{(g)}n^{5(g-1)/2-1}\gamma^n n!
\end{equation}
where $c^{(g)}$ is a positive constant and $\gamma$ is as before.


For $\mu\in(1,3)$, the number $a_{n,m}^{(g)}$ of graphs with $n$
vertices and $m=\lfloor \mu n \rfloor$ edges that can be embedded in
$\mathbb{S}_g$ satisfies
$$
a_{n,m}^{(g)} \sim \, c_{\mu}^{(g)}n^{5g/2-4}(\gamma_{\mu})^n n!\ \
when \ n\to\infty, 
$$
where $c_{\mu}^{(g)}$ is a positive constant and $\gamma_{\mu}$ is
as before.
\end{theorem}


We also prove an analogous result for non-orientable surfaces. Let
$ \mathbb{N}_h$ be the non-orientable surface of genus $h$, that
is, a sphere with $h$ crosscaps.

\begin{theorem}\label{theo:main2}
For $h\geq 1$, the number $b_n^{(h)}$ of graphs with $n$ vertices
that can be embedded in the non-orientable surface $\mathbb{N}_h$
of genus $h$ satisfies
$$
b_n^{(h)} \sim \,\tilde{c}^{(h)}n^{5(h-2)/4-1}\gamma^n n!
$$
where $\tilde{c}^{(h)}$ is a positive constant and $\gamma$ is as before.

For $\mu\in(1,3)$, the number $b_{n,m}^{(h)}$ of graphs with $n$
vertices and $m=\lfloor \mu n \rfloor$ edges that can be embedded in
$\mathbb{N}_h$ satisfies
$$
b_{n,m}^{(h)} \sim \, \tilde{c}_{\mu}^{(h)}n^{5h/4-4}(\gamma_{\mu})^nn!\ \
when \ n\to\infty, 
$$
where $\tilde{c}_{\mu}^{(h)}$ is a positive constant and $\gamma_{\mu}$
is as before.
\end{theorem}

In theory, the constants $c^{(g)}$ can be computed via non-linear recursions. Indeed, our computations relate $c^{(g)}$ to the asymptotic number of maps embedded on the surface $\mathbb{S}_g$, and weighted by their number of vertices, which are shown to obey such recursions in~\cite{BeCaRi93}. However, these recursions are so intricate that in practice is not easy to compute even the first few of these numbers. The same is true for the numbers $\tilde{c}^{(h)}$.
The case $\mu<1$ of the above theorems has been treated in the planar case in~\cite{gerkeetal}, and gives results of a different nature.
Also, it was shown in~\cite{gimeneznoy} that $\lim_{\mu\rightarrow 3^-} \gamma_{\mu} = \frac{256}{27}$, which is the exponential growth rate of triangulations embedded on a fixed surface. \\

There are three main ingredients in our proof. The first one is
the theory of map enumeration, started by Tutte in his pioneering
work on planar maps, and extended later by Arqu\`es, Bender,
Canfield, Gao, Richmond, Wormald,  and others to arbitrary
surfaces. Our main references in this context are \cite{BeCa86},
\cite{BeCaRi93} and \cite{BeGaRiWo96}. In particular, Bender and
Canfield \cite{BeCa86} showed that the number of rooted maps with
$n$ edges embeddable in $\mathbb{S}_g$ grows asymptotically as
$$
    t_g n^{5(g-1)/2} 12^n,
$$
for some constant $t_g > 0$. If we compare it with Estimate
(\ref{eq:asympt-orient}) in Theorem~\ref{theo:main}, we see that
they are very similar. This is no coincidence, since our counting
of graphs relies in a fundamental way on the counting of maps (the
extra factor of $n$ in the estimate occurs because maps are
rooted, and the absence of a factorial term is because they are
unlabelled).

The second ingredient is topological graph theory, in particular
the concept of face-width, which measures in some sense the local 
planarity of an embedding of a graph in a surface. According to
Whitney's theorem, planar 3-connected graphs have a unique
embedding in the sphere, but this is not true for arbitrary
surfaces. The key result  is that a 3-connected graph with large
enough face-width has a unique embedding \cite{Mohar}. It turns
out  that almost all 3-connected graphs have large face-width and,
as a consequence, the asymptotic enumeration of 3-connected
\emph{graphs} in a surface can be reduced to the enumeration of
3-connected \emph{maps}. In order to enumerate 3-connected maps of
genus $g$ we start from the known enumeration of maps of genus $g$
\cite{BeCa86,BeCaRi93} via associated quadrangulations.

The final step is to go from 3-connected graphs in a surface to
2-connected, connected and finally arbitrary graphs. Again, the
face-width plays the  main role in this reduction. A result of
Robertson and Vitray \cite{RoVi90} says that if a connected graph
$G$ of genus $g$ has face-width at least two, then $G$ has a
unique block of genus $g$ and the remaining blocks are planar. A
similar result holds for 2-connected graphs and 3-connected
components. As a consequence, the asymptotic enumeration of graphs
of genus $g$ can be reduced  to the planar case, which was
completely solved in \cite{gimeneznoy}.

There is a fundamental difference between the planar and
non-planar cases. For planar graphs we have at our disposal
\textit{exact} counting generating functions, defined through
functional and differential equations~\cite{gimeneznoy}. The
reason is precisely that 3-connected graphs have a unique
embedding, and there is a bijection with 3-connected maps, for
which we know the exact generating function. For higher surfaces
this is not the case and we have to \textit{approximate} the
counting series. If $f(x)$ is the generating function of interest,
we find series $f_1(x)$ and $f_2(x)$ that are computable and whose
coefficients have the same leading asymptotic estimates, and such
that $f_1(x)$ dominates $f(x)$ coefficient-wise from below, and
$f_2(x)$ from above. If we can estimate the coefficients of the
$f_i(x)$, then we can estimate those of $f(x)$. A key argument in
the proofs is the following. If a graph $G$ of genus $g$  has a
short non-contractible cycle $C$, then cutting $G$ along $C$
produces either a graph of genus $g-1$, or two graphs whose genera
add up to $g$. In either case, induction on the genus $g$ shows
that there are few such graphs and that the probability of a graph
having a short non-contractible cycle tends to zero as the size of
the graph grows. This is precisely where approximate counting
series come into play.

The paper is structured as follows. Section~\ref{sec:sing}
contains the analytic tools used in the paper. In
Section~\ref{sec:maps} we study the enumeration of a family of
maps, near-irreducible quadrangulations, closely related to
3-connected graphs. In order to prove our main results, we need a
careful analysis of quadrangulations according to their
face-width. The proof of Theorems~\ref{theo:main} and
\ref{theo:main2} comes in Section~\ref{sec:graphs}, again using
results on face-width. On the way, we also characterize the asymptotic
number of 2-connected graphs of genus $g$ (Theorem~\ref{theo:2conn}). 
In Section~\ref{sec:parameters} we apply the
machinery we have developed to analyze random graphs of genus $g$.

After we announced the result of Theorem~\ref{theo:main} at the conferences \emph{AofA'09} and \emph{Random Graphs and Maps on Surfaces (Institut Henri Poincar\'e, 2009)}, E.~Bender and Z.~Gao informed us of their ongoing project to address the same problem, with slightly different methods. They have recently issued a draft of their approach~\cite{BeGao}. As far as we can tell, the main difference with our paper will be in the way of handling the singular analysis of bivariate generating functions.

\section{Singularity analysis}\label{sec:sing}

In our work  singular expansions are expressed in terms both of square-roots
\emph{and\/} logarithms (indeed the leading terms  
in the series counting graphs embeddable on the torus and on 
the Klein bottle involve logarithms).

We define a set $D$ of ordered pairs of integers as follows:
$$
D = \mathbb{Z}^2 \,\backslash\ \{(2a, 0)\ ;\ a\in\mathbb{Z}\}.
$$
Let  $f(x)$ be a series with non-negative coefficients and finite
positive radius of convergence $\rho$. We use throughout the
following notations:
$$
X=\sqrt{1-x/\rho}, \quad L=\ln(1-x/\rho).
$$
Monomials of the form $X^aL^b$, with $(a,b)\in D$, are totally
ordered as follows (essentially a lexicographic order):
$X^aL^b<X^{a'}L^{b'}$ if either $a'<a$ or $\{a'=a, b'>b\}$.
An \emph{expansion-series} is a series $h(X,L)$ of the form
$$
h(X,L)=X^d\, \th(X,L),
$$
where $d \in\mathbb{Z}$, $\th(X,L)$ is analytic at $(0,0)$, and
$\th(X,L)$ is of the form $h_0(X)+h_1(X)L+\ldots+h_c(X)L^c$ (i.e., a polynomial in $L$)
with the $h_i$'s analytic at $0$.
The \emph{type} of $h(X,L)$ is the
largest $(a,b)\in D$ such that $[X^aL^b]h(X,L)\neq 0$.

The series $f(x)$ is said to admit a \emph{singular expansion} of
type $(a,b)\in D$ if, in a complex neighbourhood of $\rho>0$
(except on  $x-\rho\in\mathbb{R}_+$):
\begin{equation}\label{eq:sing}
f(x)=h(L,X),
\end{equation}
where $h(L,X)$ is an expansion-series of type $(a,b)$ (we also say
of type $X^aL^b$). The singular expansion is called \emph{strong}
if $f(x)$ is analytically continuable to a complex domain of the
form $\Omega=\{|x|\leq
\rho+\delta\}\backslash\{x-\rho\in\mathbb{R}_+\}$ for some
$\delta>0$. We consider mostly singular expansions of type $X^a$,
with $a\in\mathbb{Z}\backslash (2\mathbb{Z}_{\geq 0})$
---such singular expansion are called \emph{of order $a/2$}--- and
singular expansions of type $X^{2a}L$, with $a\in\mathbb{Z}_{\geq
0}$ ---called \emph{of order $a$}. In particular, if the order is
$0$, a singular expansion has $L$ as leading monomial.
With this definition, the order $\alpha$ can take values that are either
integers or half-integers (nonnegative or negative), which
we write as $\alpha\in\tfrac1{2}\mathbb{Z}$. 

\cacher{
\begin{equation}
f(x)=g(X)\log(Z)+h(X),\ \ \mathrm{with}\ Z=(1-x/\rho),\ X=\sqrt{Z},
\end{equation}
where $g(Z)$ and $h(X)$ are as follows:
\begin{itemize}
\item
If $\alpha$ is a non-negative integer, then $g(X)$ is of the form
$g(X)=Z^{\alpha}\tg(X)$, with $\tg(X)$ analytic at $0$ and
$\tg(0)\neq 0$; and $h(X)$ is analytic at $0$ and
$[X^{2k+1}]h(X)=0$ for $k< \alpha$.
\item
If $\alpha$ is a positive half-integer, then $[X^{\beta}]g(X)=0$
for $\beta\leq 2\alpha$ and $h(X)$ is of the form
$h(X)=Z^{\alpha}\th(X)$, with $\th(X)$ analytic at $0$ and
$\th(0)\neq 0$.
\item
If $\alpha$ is negative, then $g(X)$ is analytic at $0$, and
$h(X)$ is of the form $h(X)=Z^{\alpha}\th(X)$, with $\th(X)$
analytic at $0$ and $\th(0)\neq 0$.
\end{itemize}
}

From a singular expansion of  a series $f(x)$ one can obtain
automatically an asymptotic estimate for the coefficients
$[x^n]f(x)$.

\begin{theorem}[Transfer theorem~\cite{FlaSe}]\label{theo:transfer}
Let $f(x)$ be a series with non-negative coefficients that admits
a strong singular expansion of order $\alpha$ around its radius of
convergence $\rho>0$.
 Then $f_n=[x^n]f(x)$ satisfies
$$
f_n\sim c\ \!\rho^{-n}n^{-\alpha-1},
$$
where $c$ depends explicitly on the ``dominating'' coefficient of
the expansion-series $h(X,L)$.
\end{theorem}

Let us now extend the concept to a bivariate series $f(x,u)$ with
non-negative coefficients. For $u>0$, let $\rho(u)$ be the radius
of convergence of $x\mapsto f(x,u)$. The function $\rho(u)$ is
called the \emph{singularity function} of $f(x,u)$ with respect to
$x$, and each point $(x_0,u_0)$ such that $x_0=\rho(u_0)$ is
called a \emph{singular point} of $f(x,u)$. We use the notations
$$
X=\sqrt{1-x/\rho(u)},\quad  L=\ln(1-x/\rho(u)),\quad  U=u-u_0.
$$
A trivariate series $h(X,L,U)$ is called an \emph{expansion
series} if it is of the form
$$
h(X,L,U)=X^d\, \th(X,L,U),
$$
with $d\in\mathbb{Z}$, $\th(X,L,U)$ analytic at $(0,0,0)$ and of
polynomial dependence in $L$, that is, $\th(X,L,U)$ is of the form
$h_0(X,U)+h_1(X,U)L+\ldots+h_c(X,U)L^c$, with the $h_i$'s analytic at $(0,0)$. If there is some
$(a,b)\in D$ such that $[X^aL^b]h(X,L,U)$  is non-zero at $U=0$
and
 $[X^{a'}L^{b'}]h(X,L,U)=0$ for $(a',b')>(a,b)$, then
the expansion series is said to be \emph{of type $(a,b)$}.

Let $(x_0,u_0)$ be a singular point of $f(x,u)$ such that
$\rho'(u_0)\neq 0$ and such that $\rho(u)$ is analytically
continuable to a complex neighbourhood of $u_0$. Then   $f(x,u)$
is said to admit a \emph{singular expansion} of type $(a,b)\in D$,
or of type $X^aL^b$, if  in a complex neighbourhood of $(x_0,u_0)$
(except on  the set $1-x/\rho(u)\in\mathbb{R}_{\leq 0}$) we have
\begin{equation}\label{eq:sing_biv}
f(x,u)=h(X,L,U),
\end{equation}
where $h(X,L,U)$ is an expansion series of type $(a,b)$. The
singular expansion is called \emph{strong} if $f(x,u)$ is
analytically continuable to a complex domain of the form
$\Omega=\{|x|\leq x_0+\delta, |u|\leq
u_0+\delta\}\backslash\{1-x/\rho(u)\in\mathbb{R}_{\leq 0}\}$ for some
$\delta>0$.
 Again we consider mostly singular expansions of type $X^a$, with
$a\in\mathbb{Z}\backslash (2\mathbb{Z}_{\geq 0})$
---called \emph{of order $a/2$}--- and
singular expansions of type $X^{2a}L$, with $a\in\mathbb{Z}_{\geq
0}$ ---called \emph{of order $a$}. A singular expansion is called
\emph{log-free} if $h$ does not involve $L$, that is, $\partial_L
h(X,L,U)=0$. Singular expansions of order $1/2$ are called
\emph{square-root} singular expansions.

At some points we will need the following easy lemma:

\begin{lemma}\label{lem:sing_prod}
If two series $f(x,u)$, $g(x,u)$ with non-negative coefficients
admit singular expansions of orders $\alpha\leq\alpha'$
around $(x_0,u_0)$, then the product-series $h(x,u)=f(x,u)\cdot g(x,u)$ admits a 
singular expansion at $(x_0,u_0)$, 
\begin{itemize}
\item
of order $\alpha+\alpha'$ if $\alpha<0$ and $\alpha'<0$,
\item
of order $\alpha$ if $\alpha<0<\alpha'$.
\end{itemize}
\end{lemma}

\paragraph{Remark.}
Our definition of singular expansions may seem a bit technical,
but it meets the following convenient requirements: they are
closed under derivation and integration (see
Lemma~\ref{lem:integrate} in this section), closed under product,
and they include the basic singular functions $\sqrt{1-x/\rho}$
which typically appear in series counting maps, and
$\ln(1-x/\rho)$ which appear in series counting unrooted graphs in
the torus and in the Klein bottle.

\medskip
From bivariate singular expansions, it is possible to get
asymptotic estimates for the coefficient $f_n=[x^n]f(x,1)$, (just
apply Theorem~\ref{theo:transfer} to $f(x,1)$), but also to get
asymptotic estimates according to two parameters.

\begin{theorem}[Local limit theorem~\cite{FlaSe}]
Suppose that $f(x,u)=\sum_{n,m}f_{n,m}x^nu^m$ admits a strong singular
expansion of order $\alpha$ around each singular point
$(\rho(u_0),u_0)$, and let $\rho(u)$ be the singularity function
of $f(x,u)$. Suppose that the function $\mu(u)=-u\rho'(u)/\rho(u)$
is strictly increasing, and define $a=\mu(0^+)$,
$b=\lim_{u\to\infty} \mu(u)$. For each $\mu\in(a,b)$, let $u_0$
be the positive value such that $\mu(u_0)=\mu$, and let
$\gamma(\mu)=1/\rho(u_0)$. Then there exists a constant $c(\mu)>0$
such that
  $$
  f_{n,m}\sim c(\mu)\ \!n^{-\alpha-3/2}\gamma(\mu)^n\ \ \mathrm{when}\ n\to\infty\ \mathrm{and}\ m=\lfloor \mu n \rfloor.
  $$
\end{theorem}

\cacher{ Then $f(x,u)$ is said to admit a singular expansion of
order $\alpha\in\tfrac1{2}\mathbb{Z}$ if the following expansion
holds in a complex neighbourhood of $(x_0,u_0)$ (except on the
part where $x-\rho(u)\in\mathbb{R}_+$):
\begin{equation}
f(x)=g(X,U)\log(Z)+h(X,U),\ \mathrm{with}\ Z=(1-x/\rho(u)),\ X=\sqrt{Z},\ U=u-u_0,
\end{equation}
where $g(X,U)$ and $h(X,U)$ are as follows:
\begin{itemize}
\item
If $\alpha$ is a non-negative integer, then $g(Z,U)$ is of the form
$g(Z)=Z^{\alpha}\tg(X,U)$, with $\tg(X,U)$ analytic at $(0,0)$ and $\tg(0,0)\neq 0$;
and $h(X,U)$ is analytic at $0$ and  $[X^{2k+1}]h(X,U)=0$
for $k< \alpha$.
\item
If $\alpha$ is a positive half-integer, then $[X^{\beta}]g(X,U)=0$
for $\beta\leq 2\alpha$  and $h(X,U)$ is of the form
$h(X,U)=Z^{\alpha}\th(X,U)$, with $\th(X,U)$ analytic at $(0,0)$
and $\th(0,0)\neq 0$.
\item
If $\alpha$ is negative, then $g(X,U)$ is analytic at $(0,0)$, and $h(X,U)$
is of the
form $h(X,U)=Z^{\alpha}\th(X,U)$, with $\th(X,U)$ analytic at $(0,0)$
and $\th(0,0)\neq 0$.
\end{itemize}

The singular term is said to have \emph{no logarithmic term} if
$g(X,U)=0$. In the special case $\alpha=1/2$, the singular
expansion is called a \emph{square-root singular expansion}.
 }

\begin{lemma}[Exchange of variables, from~\cite{DrGiNo08}]\label{lem:exch}
With $f(x,u)$ and $\rho(u)$ as in the previous lemma, assume
$(x_0,u_0)$ is a singular point and  $\rho'(u_0)\ne0$.

If $f(x,u)$ admits a singular expansion of order $\alpha$ around
$(x_0,u_0)$ with leading variable $x$, then $f(x,u)$ also admits a
singular expansion of order $\alpha$ around $(x_0,u_0)$ with
leading variable $u$. In addition, if the singular expansion in $x$ is strong,
then the singular expansion in $u$ is also strong.
 \end{lemma}
\noindent\emph{Sketch of proof.} Let $R(x)$ be the local inverse
of $\rho(u)$ at $u_0$. Then, by the Weierstrass preparation
theorem, there exists $Q(x,u)$ analytic at $(x_0,u_0)$, with
$Q(x_0,u_0)\neq 0$, such that $x-\rho(u)=(u-R(x))\cdot Q(x,u)$ in
a neighbourhood of $(x_0,u_0)$. Replacing $x-\rho(u)$ by
$(u-R(x))Q(x,u)$ and rearranging terms, one obtains a singular
expansion with leading variable $u$; and this singular expansion is 
strong if the one in $x$ is strong (indeed, by definition, the domain $\Omega$
of analytic continuation required by strong expansions is symmetric
in $x$ and $u$).
 \hfill\qed

\medskip
\noindent By Lemma~\ref{lem:exch}, we can omit to mention which is
the leading variable in a singular expansion, since the other
variable could be chosen as well.

We need the following two technical lemmas to manipulate singular
expansions.

\begin{lemma}[Derivation and integration, from~\cite{DrGiNo08}]\label{lem:integrate}
Let $f(x,u)$ be a series that admits a singular expansion of order
$\alpha\in\tfrac{1}{2}\mathbb{Z}$ around $(x_0,u_0)$. Then the
series $\frac{\partial f}{\partial x}(x,u)$ admits a singular
expansion of order $\alpha-1$ around $(x_0,u_0)$ with the same
singularity function; and  the series
$\int_{0}^xf(s,u)\mathrm{d}s$ admits a singular expansion of order
 $\alpha+1$ around $(x_0,u_0)$ with the same singularity function.

Moreover, if the expansion of $f(x,u)$ is strong, then so are the ones of $\frac{\partial f}{\partial x}(x,u)$ and $\int_{0}^xf(s,u)\mathrm{d}s$.
\end{lemma}
\begin{proof}
We use again the notations $X=\sqrt{1-x/\rho(u)}$,
$L=\ln(1-x/\rho(u))$. The main point is that, for 
$a\in 2\mathbb{Z}_{>0}$, $\partial_x X^{a}L$ has
dominating monomial $X^{a-2}L$, $\partial_xL$ 
 has dominating monomial $X^{-2}$, and for $a\in\mathbb{Z}\backslash 2\mathbb{Z}_{\geq 0}$,   
$\partial_x X^a$ has dominating
monomial $X^{a-2}$.  We refer the reader to~\cite{DrGiNo08} for a
detailed proof in the case of a square-root singular expansion,
the arguments of which can be extended directly to any order.
\end{proof}

 \begin{lemma}[Critical composition schemes]\label{lem:comp}
Let $M(x,u)$, $H(x,u)$, $C(t,u)$ be non-zero series with
non-negative coefficients linked by the equation
 $$
 M(x,u)=C(H(x,u),u).
 $$
Assume that $H(x,u)$ has a log-free strong singular expansion of order
$3/2$ around $(x_0,u_0)$, and let $\rho(u)$ be the singularity
function of $H(x,u)$. Assume also that  $C(t,u)$ has a strong singular
expansion of order $\alpha \leq 0$ at $(t_0,u_0)$, with
$t_0=H(x_0,u_0)$, and the singularity function $R(u)$ of $C(t,u)$
coincides with $H(\rho(u),u)$ (this means that the composition is
\emph{critical}).

Then $M(x,u)$ has a strong singular expansion of order $\alpha$
around $(x_0,u_0)$ with singularity function $\rho(u)$.
 \end{lemma}
 \begin{proof}
Since $C(t,u)$ has a singular expansion of order $\alpha$ around
$(t_0,u_0)$, the series $M(x,u)=C(H(x,u),u)$ has an expansion of
the form
 $$
M(x,u)\!=\!\bh(\bX,\bL,U),\quad  with\ t\!=\!H(x,u),\
T\!=\!1-\frac{t}{R(u)},\ \bX\!=\!\sqrt{T},\  \bL\!=\!\ln(T),\
U\!=\!u-u_0,
 $$
with $\bh$ an expansion-series of type $(2\alpha,0)$ if
$\alpha\in\tfrac1{2}\mathbb{Z}\backslash\mathbb{Z}_{\geq 0}$, or
of type $(2\alpha,1)$ if $\alpha\in\mathbb{Z}_{\geq 0}$. Moreover,
since $H(x,u)$ has a log-free singular expansion of order $3/2$,
we have
 $$
 T=Z\cdot\lambda(X,U),\ \mathrm{with}\ Z=1-\frac{x}{\rho(u)},\ X=\sqrt{Z},
 $$
where $\lambda(X,U)$ is analytic and non-zero in a neighbourhood
of $(0,0)$, which implies that $\bX=X\cdot\mu(X,U)$ and $\bL=L\cdot \nu(X,U)$
with $\mu(X,U)$ and $\nu(X,U)$ analytic and non-zero in a neighbourhood of $(0,0)$. 
Hence, by easy rearrangements, $M(x,u)$ can be written
as
$$
 M(x,u)\!=\!\ h(X,L,U),\ \mathrm{with}\ X=\sqrt{Z},\ L=\ln(Z),\ U\!=\!u-u_0,
 $$
 where $h(X,L,U)$ is an expansion-series of the same type as $\bh(X,L,U)$.

 \cacher{
Since $C(t,u)$ has a singular expansion of order $\alpha$ around
$(t_0,u_0)$, the series $M(x,u)=C(H(x,u),u)$ has an expansion of
the form
 $$
 M(x,u)\!=\!g(S,U)\ln(T)+h(S,U),\ \mathrm{with}\
 t\!=\!H(x,u),\ T\!=\!1-\frac{t}{R(u)},\ S\!=\!\sqrt{T},\ U\!=\!u-u_0,
 $$
where $g$ and $h$ meet the conditions of a singular expansion of
order $\alpha$, as listed in the definition given above. Moreover,
since $H(x,u)$ has a singular expansion of order $3/2$ with no
logarithmic term, one has
 $$
 T=Z\cdot\lambda(X,U),\ \mathrm{with}\ Z=1-\frac{x}{\rho(u)},\ X=\sqrt{Z},
 $$
and where $\lambda(X,U)$ is analytic and non-zero in a
neighbourhood of $(0,0)$. Hence, by easy calculations, $M(x,u)$
can be written as
$$
M(x,u)=\tg(X,U)\ln(Z)+\th(X,U),
$$
where $\tg$ and $\th$ meet the same conditions as $g$ and $h$,
respectively. Thus $M(x,u)$ has a singular expansion around
$(x_0,u_0)$ of the same order as the expansion of~$C(t,u)$. }

It remains to prove that the singular expansion of $M(x,u)$ is
strong. For $z\in\mathbb{C}$ and $\epsilon>0$, denote by
$\cB(z,\epsilon)$ the open ball of center $z$ and radius
$\epsilon$. Since $M(x,u)$ has a singular expansion around
$(x_0,u_0)$ and has nonnegative coefficients, 
$M(x,u)$ is already analytically continuable to a
domain of the form $D_1\cup D_2$, with
$$
D_1=\{(x,u)\in\mathbb{C}^2\,:\, |x|<x_0,\ |u|<u_0\},$$ and
$$
D_2:\{(x,u)\in\mathbb{C}^2 \,:\, x\in \cB(x_0,\epsilon),\ u\in
\cB(u_0,\epsilon),\ x-\rho(u)\notin \mathbb{R}_+\},
$$
for some $\epsilon>0$.
So we just have to show that $M(x,u)$ is analytically continuable to the domain
$$
D_3=\{(x,u)\in\mathbb{C}^2\,:\, x_0\leq |x|<x_0+\delta,\ u_0\leq
|u|<u_0+\delta\}\backslash (\cB(x_0,\epsilon)\times
\cB(u_0,\epsilon)),
$$
for some $\delta>0$ to be adjusted. Note that, for $\delta$
sufficiently small (compared to $\epsilon$), there exists $\phi>0$
such that, for $(x,u)\in D_3$, the arguments of $x$ and $u$ are in
$(\phi,2\pi-\phi)$, that is, $x$ and $u$ are \emph{away} from the
positive real axis. Since $H(x,u)$ has nonnegative coefficients,
for $\delta$ small enough, we have $|H(x,u)|<t_0-\eta$ for some
$\eta>0$, where $t_0=H(x_0,u_0)$). Since $C(t,u)$ has non-negative
coefficients, its singularity function $R(u)$ satisfies $R(u)\geq
R(u_0+\delta)$ for $|u|\leq u_0+\delta$. Reducing again $\delta$
so that $R(u_0+\delta)>t_0-\eta/2$, we conclude that
$|H(x,u)|<R(u)$ for $(x,u)\in D_3$. Hence $C(t,u)$ is analytic at
$(H(x,u),u)$, and so $M(x,u)$ is analytic at any $(x,u)\in D_3$.
\end{proof}

We need some further definitions. Given two series $f_1$ and $f_2$
in the same number of variables, write $f_1\preceq f_2$ if
$f_2-f_1$ has non-negative coefficients. A bivariate series
$f(x,u)$ is said to be of  \emph{approximate-singular order}
$\alpha$ at 
 $(x_0,u_0)$ if there are two series $D(x,u)$ and $E(x,u)$ such that
$$
D(x,u)-E(x,u)\preceq f(x,u)\preceq D(x,u),
$$
where $D(x,u)$ has a singular expansion of order $\alpha$ and
$E(x,u)$ has a singular expansion of order $\beta>\alpha$ around
$(x_0,u_0)$, and $D(x,u)$ and $E(x,u)$ have the same singularity
function. This common singularity function is called the
\emph{singularity function} of $f(x,u)$. The function $f(x,u)$ is said to be
of \emph{strong} approximate-singular order $\alpha$ if the singular
expansions of $D(x,u)$ and $E(x,u)$ are strong. 

Since the singular order of $D(x,u)$ dominates the singular order
of $E(x,u)$, the asymptotic estimate of the coefficients of
$f(x,u)$ is dictated by the estimate of $D(x,u)$, and we have the
following basic result.

\begin{corollary}\label{coro:transfer}
Let $f(x,u)$ be a series of strong approximate-singular order
$\alpha\in\tfrac1{2}\mathbb{Z}$ around each singular point of
$f(x,u)$, and let  $\rho(u)$ be the singularity function of
$f(x,u)$.  Then $f_n=[x^n]f(x,1)$ satisfies
$$
f_n\sim c\ \!n^{-\alpha-1}\gamma^n,
$$
where $c>0$ is a constant and $\gamma=1/\rho(1)$.

Assume that $\mu(u)=-u\rho'(u)/\rho(u)$ is strictly increasing,
and define $a=\mu(0^+)$, $b=\lim_{u\to\infty} \mu(u)$. For each
$\mu\in(a,b)$, let $u_0$ be the positive value such that
$\mu(u_0)=\mu$, and let $\gamma(\mu)=1/\rho(u_0)$. Then there
exists a constant $c(\mu)>0$ such that
  $$
  f_{n,\lfloor \mu n \rfloor}\sim c(\mu)\ \!n^{-\alpha-3/2}\gamma(\mu)^n\ \ \mathrm{when}\ n\to\infty.
  $$
\end{corollary}

We also have a corollary regarding approximate-singular orders of
critical compositions. The following result follows directly from
Lemma~\ref{lem:comp}, but notice that the statement is not the
same, since we are now dealing with approximate-singular orders
(which is precisely  what is needed later on).

\begin{corollary}\label{coro:comp}
Let $M(x,u)$, $H(x,u)$, $C(t,u)$ be non-zero series with
non-negative coefficients linked by the equation
 $$
 M(x,u)=C(H(x,u),u).
 $$
Assume that $H(x,u)$ has a log-free strong singular expansion of order
$3/2$  around $(x_0,u_0)$ ---call $\rho(u)$ the singularity
function of $H(x,u)$---, $C(t,u)$ is of strong approximate-singular order
$\alpha$ at $(t_0,u_0)$, with $t_0=H(x_0,u_0)$, and the
singularity function $R(u)$ of $C(t,u)$ coincides with
$H(\rho(u),u)$ (the composition is \emph{critical}).

Then $M(x,u)$ is of strong approximate-singular order $\alpha$
around $(x_0,u_0)$ with singularity function~$\rho(u)$.
 \end{corollary}

\section{Maps of genus $g$ via quadrangulations}\label{sec:maps}

We recall that a map is a graph $G$ embedded in a surface $S$ in
such a way that the faces are homeomorphic to topological disks.
If $S=\mathbb{S}_g$ we say that the map is of genus $g$. We have
a similar definition for the non-orientable surface $
\mathbb{N}_h$ of genus $h$. Loops and multiple edges are allowed.
A map is \emph{rooted} if it has a marked edge that is given a
direction. If $g=0$ (planar maps), the face on the right of the root is called
the \emph{outer face}, and the other faces are called \emph{inner faces}. 
Vertices and edges incident to the outer face are said to be \emph{outer} vertices (edges, resp.),
the other ones are called \emph{inner} vertices (edges, resp.). 

In this section we obtain the asymptotic enumeration of a family
of maps $\cS_g$ that is closely related to 3-connected graphs of
genus $g$. To enumerate $\cS_g$ we follow the scheme of Bender et
al.~\cite{BeGaRiWo96}, and use a classical correspondence between
maps and quadrangulations (maps in which every face has degree
four). We decompose quadrangulations successively at contractible
2-cycles and then at contractible 4-cycles. Along this way we
encounter near-simple and near-irreducible quadrangulations.

However, in our treatment, quadrangulations are labelled at edges
and are allowed to be unrooted (classically, in map enumeration
maps are unlabelled and rooted). Since each rooted quadrangulation
with $2n$ edges corresponds to exactly $(2n-1)!$ labelled
quadrangulations, the ordinary generating function of rooted
objects is given by the derivative of the exponential generating
function of the labelled ones. We believe this approach yields
simplified decompositions and equations. Note that the
corresponding generating functions have to be exponential to take
into account the labelling of the edges. Subsections
\ref{subsec:maps} and \ref{subsec:quadrang} contain the basic
concepts and definitions. In Subsection \ref{subsec:quad-singular}
we obtain singular expansions that are essential for the
enumeration of graphs in Section~\ref{sec:graphs}.

\subsection{Maps and quadrangulations, face-width and
edge-width}\label{subsec:maps}
A map is \emph{bipartite} if there
is a partition of the vertices into black and white vertices such
that each edge connects a black vertex with a white vertex.
Whenever we work with bipartite maps we consider the partition as
given. A quadrangulation is a map such that all faces have degree
four. In this article we consider only bipartite quadrangulations.
(Note that in positive genus not all quadrangulations are
bipartite;  for instance, a toroidal grid $C_p \times C_q$ is a
quadrangulation but it has odd cycles if either $p$ or $q$ are
odd.) Here, by a quadrangulation we always mean a
\emph{bipartite} quadrangulation.

We recall here a well-known correspondence, for any fixed genus
$g\geq 0$, between maps and quadrangulations. Given a map $M$,
whose vertices are depicted as black, consider the following
operations:
\begin{itemize}
\item
insert a white vertex in each face of $M$,
\item
for each corner $\theta$ of $M$, with $v$ the incident vertex and
$f$ the incident face of $\theta$, insert an edge that connects
$v$ to the white vertex corresponding to $f$. Such an edge is
called a \emph{corner-edge}.
\end{itemize}
The embedded graph $Q$ consisting of all vertices (black and
white) and of all corner-edges is a (bipartite) quadrangulation,
which is called the \emph{quadrangulation} of $M$. A correspondence
between the parameters of $M$ and the parameters of $Q$ is listed
in the following table:

\begin{center}
\begin{tabular}{|r|l|}
\hline
$M$ & $Q$ \\
\hline
\hline
corner & edge\\
edge& face\\
\phantom{black} vertex& black vertex\\
face& white vertex\\
\hline
\end{tabular}
\end{center}

Let us now define the face-width and the edge-width of a map $M$
in a surface $S$ of genus $g\geq 0$. The \emph{face-width}
$\fw(M)$ of $M$ is the minimum number of intersections of $M$ with
a simple non-contractible curve $C$ on $S$. It is easy to see that
this minimum is achieved when $C$ meets $M$ only at vertices. The
\emph{edge-width} $\ew(M)$ of $M$ is the minimum length of a
non-contractible cycle of $M$. (For genus $0$, the edge-width
and the face-width are defined to be $+\infty$ since every curve
is contractible.) Let $Q$ be the quadrangulation of $M$. Since a closed curve
such that $|C\cap M|=\ell$ yields a non-contractible cycle of
$Q$ of length $2\ell$ and vice versa, we have~(see
\cite{Mohar})
\begin{equation}\label{eq:fwew}
\ew(Q)=2\ \!\fw(M).
\end{equation}

\subsection{Families of quadrangulations and links between them}\label{subsec:quadrang}
Again, all quadrangulations are assumed to be bipartite. In
addition, a quadrangulation is assumed to have its $2m$ edges
labelled $\{1,2, \dots, 2m\}$. This is not the standard procedure
with maps but, as explained earlier, in our context it is easier
to work with unrooted maps, and in this case we must label the
edges in order to avoid caring about automorphisms.

A quadrangulation is called \emph{near-simple} if it has no
contractible 2-cycle, and is called \emph{simple} if it has no
2-cycle at all, that is, no multiple edges.
A quadrangulation is called \emph{near-irreducible} if it is
simple and every contractible 4-cycle delimits a face, and is
called \emph{irreducible} if it is simple and every 4-cycle delimits a
face. Note that a near-simple quadrangulation $Q$ is simple if and
only if $\ew(Q)>2$, and a near-irreducible quadrangulation $Q$ is
irreducible if and only if $\ew(Q)>4$. For $g\geq 0$, the families
of (bipartite) quadrangulations, near-simple quadrangulations, and
near-irreducible quadrangulations of genus $g$ are denoted,
respectively, by $\cQg$, $\cRg$ and $\cSg$. If  $\cF$ is a family
of  quadrangulations with labelled edges, let $\cF[n,m]$ be the
set of elements in $\cF$ with $n$ black vertices and $m$ faces,
and let
$$
F(x,u)=\sum_{n,m}\bigl|\cF[n,m]\bigr|\,x^n\frac{u^{m}}{(2m)!}.
$$
In other words, $F(x,u)$ enumerates the family $\cF$ according to
black vertices and faces, and the factor $(2m)!$ in the
denominator takes account of the fact that the $2m$ edges are
labelled. We also define $F'(x,u)$ as the series $\partial_u
F(x,u)$, that is, differentiation  with respect to the second
variable.

Let us recall how one extracts, for $g>0$, a near-simple core
$R\in\cRg$ from $Q\in\cQg$, and a near-irreducible core $S\in\cSg$
from $R\in\cRg$, as described in~\cite{BeGaRiWo96}. The
contractible 2-cycles of $Q\in\cQg$ are ordered by inclusion, that
is,  given two 2-cycles $c$ and $c'$, one has $c\leq c'$ if the
disc inside $c$ is included in the disc inside $c'$. The
\emph{near-simple core} $R$ of $Q$ is obtained by collapsing all
maximal contractible 2-cycles into edges. Conversely, every
$Q\in\cQg$ is obtained from some $R\in\cRg$, where each edge is
possibly ``opened'' into a 2-cycle that is to be filled with a (rooted) 
planar quadrangulation. Therefore, for $g>0$ we have
\begin{equation}\label{eq:QR}
\Qg(x,u)=\Rg(x,V(x,u)),\ \ \ \mathrm{where}\
V(x,u)=u(1+\tfrac{2u}{x}Q_0\ \!\!\!'(x,u))^2.
\end{equation}
Similarly, the contractible 4-cycles of $R\in\cRg$ are partially
ordered according to the analogous inclusion relation. The
\emph{near-irreducible core} $S$ of $R\in\cR_g$ is obtained by
emptying the disc within each maximal contractible 4-cycle, that
is, the disc inside a maximal 4-cycle is replaced by a face.
Conversely, every $R\in\cRg$ is obtained from some $S\in\cSg$,
where one might insert in each face a planar simple
quadrangulation (the ``degenerate'' simple planar quadrangulation
with only 2 edges is not allowed for insertion). Therefore, for
$g>0$ we have
\begin{equation}\label{eq:RS}
\Rg(x,v)=\Sg(x,W(x,v)),\ \ \ \mathrm{where}\ W(x,v)=\frac1{x^2v}(2vR_0\ \!\!\!'(x,v)-xv-x^2v).
\end{equation}

\subsection{Singular expansions}\label{subsec:quad-singular}

Throughout this section, we associate to genus $g$ the value
\begin{equation}
\alpha=\tfrac{5}{2}(1-g).
\end{equation}

\paragraph{Rooted maps.}
For $g\geq 0$, let $\Mgr(x,u)$ be the series counting rooted maps
of genus $g$ where $x$ marks vertices and $u$ marks edges.
According to the correspondence between maps and quadrangulations,
we have
\begin{equation}
2u\Qg'(x,u)=\Mgr(x,u).
\end{equation}
Let $p=p(x,u)$ and $q=q(x,u)$ be the algebraic bivariate series
defined by the system
\begin{equation}\label{eq:systpq}
\left\{\begin{array}{rcl}
p&=&xu+2pq+p^2,   \\
q&=&u+2pq+q^2,
\end{array}\right.
\end{equation}
and let $\Delta=\Delta(p,q)=(1-2p-2q)^2-4pq$ be the Jacobian
determinant of the system. Singular points of $p(x,u)$ and $q(x,u)$
are the same: for $x_0>0$ the corresponding singular point $(x_0,u_0)$
is such that $u_0$ is the smallest positive $u$ satisfying $\Delta(p,q)=0$
where $p=p(x_0,u)$ and $q=q(x_0,u)$.  Such a pair $(x_0,u_0)$
is called a \emph{singular point} of~\eqref{eq:systpq}.

Define a bivariate series $f(x,u)$ to be \emph{$(p,q)$-rational} if it admits a rational
expression in terms of $p(x,u)$ and $q(x,u)$. Note that the elementary functions
$(x,u)\mapsto x$ and $(x,u)\mapsto u$ are $(p,q)$-rational, since
$u=q-q^2-2pq$ and $x=(p-2pq-p^2)/(q-2pq-q^2)$.  
Remarkably, the series counting rooted maps 
in any genus are $(p,q)$-rational. Precisely, 
Arqu\`es has shown in~\cite{Ar87} that 
\begin{equation}\label{eq:expMpr}
\overrightarrow{M}_0(x,u)=\frac1{xu^2}p\,q\,(1-2p-2q),
\end{equation}
a result obtained formerly by Tutte under an equivalent
parametrization \cite{TutteMaps}. In addition, Arqu\`es and
Giorgetti~\cite{ArGiFirst} have proved that, for $g>0$, there is a
bivariate polynomial $P_g(X,Y)$ such that
\begin{equation}\label{eq:expMgr}
\Mgr(x,u)=\frac{P_g(p,q)}{\Delta^{5g-3}}.
\end{equation}

\begin{lemma}
For every singular point $(x_0,u_0)$ of~\eqref{eq:systpq}, the polynomial $P_g$
in~\eqref{eq:expMgr} is non-zero at $(p_0,q_0)$, 
where $p_0=p(x_0,u_0)$ and $q_0=q(x_0,u_0)$. 
\end{lemma}
\begin{proof}
Bender and Canfield showed
in~\cite{BeCa86} that, for $g\geq 0$, $u\mapsto\Qgr(1,u)$ admits a
singular expansion of order $\alpha-1$ around $1/12$.
In~\cite[proof of Theorem 2]{BeCaRi93}, working with similar
algebraic expressions as~\eqref{eq:expMgr} they provide a sketchy
justification that all arguments in~\cite{BeCa86} can be extended
to the bivariate case. In our terminology it means that, for every singular
point $(x_0,u_0)$ of~\eqref{eq:systpq}, the series
$u\mapsto\Mgr(x_0,u)$ admits a singular expansion of order
$\alpha-1$ around $u_0$. Note that, using for instance the main
theorem in~\cite{drmota97systems}, one easily shows that
$$
\Delta(p(x_0,u),q(x_0,u))=\sqrt{U}\cdot h\big(\sqrt{U}\big),\ \mathrm{with}\
U=\sqrt{1-u/u_0},
$$
and with $h$ analytic and non-zero at $0$. Since
$\Mgr(x_0,u)=P_g(p(x_0,u),q(x_0,u))/\Delta^{5g-3}$ is of order $\alpha-1=-\tfrac1{2}(5g-3)$
at $u_0$ and since $P_g(p(x_0,u),q(x_0,u))$ converges to $P_g(p_0,q_0)$, we
conclude that $P_g(p_0,q_0)\neq 0$.
\end{proof}

\paragraph{Near-irreducible quadrangulations.}\label{sec:nearirr}
Composing Equation~\eqref{eq:QR} with Equation~\eqref{eq:RS},
we obtain the following link
between $\Qg$ and $\Sg$:
\begin{equation}\label{eq:QS}
\Qg(x,u)=\Sg(x,w),\ \ \mathrm{with}\ w=H(x,u)=W(x,V(x,u)).
\end{equation}
Differentiating this equation with respect to $u$ 
and multiplying by $2u$, we obtain
\begin{equation}\label{eq:MrS}
\Mgr(x,u)=2uH'(x,u)\cdot\Sg\ \!\!\!'(x,w),\ \ \mathrm{where}\ w=H(x,u).
\end{equation}

The series $H(x,u)$ has a combinatorial interpretation: it
enumerates rooted planar quadrangulations (the outer face is the
one to the right of the root edge) where the four outer vertices
are distinct, and where the root-edge and the opposite outer edge
are simple. These are only minor constraints on quadrangulations;
and by elementary decompositions and calculations (see for instance~\cite{MuSc68} for similar
calculations, and more recently~\cite[Sect.7]{ChFuKaSh07}), $H(x,u)$
is $(p,q)$-rational. So is $2uH'(x,u)$ because $(p,q)$-rationality
is stable under derivation, see~\cite{FuSLC}; call $A(p,q)$ the algebraic expression
such that $2uH'(x,u)=A(p(x,u),q(x,u))$. Note also the singular points of $2uH'(x,u)$
are the same as the singular points of $(p,q)$. (Indeed, $2uH'(x,u)$
has the same singular points as $H(x,u)$. The series $H(x,u)$ itself is dominated
 coefficientwise 
by the series counting rooted planar quadrangulations, whose singular points
are those of $(p,q)$, so the singular points of $H(x,u)$ can not occur
before those of $(p,q)$.)  

Define $r=r(x,w)$ and $s=s(x,w)$ as the algebraic series defined by the system

\begin{equation}\label{eq:systrs}
\left\{\begin{array}{rcl}
r&=&xw(1+s)^2,\\
s&=&w(1+r)^2.
\end{array}\right.
\end{equation}
%
\cacher{ The Jacobian determinant of this system is
$1+4xw^2(1+r)(1+s)=1+4rs/((1+r)(1+s))=(r+s-3rs)/((1+r)(1+s))$,
which cancels if and only if  $r+s-3rs=0$. Hence the singularities
of~\eqref{eq:systrs} are solutions of the equation}
%
As shown in~\cite{BeRi84}, for $x_0>0$ the radius of convergence
of $w\mapsto r(x_0,w)$ (and also of $w\mapsto s(x_0,w)$) is the
value $w_0>0$ such that
\begin{equation}
1+r+s-3rs=0.
\end{equation}
Such a solution $(x_0,w_0)$ is called a \emph{singular point} 
of~\eqref{eq:systrs}. Hence singular points of~\eqref{eq:systrs}
are the singular points of $r(x,w)$ (and also of $s(x,w)$).

\begin{lemma}
For $x_0$, let $u_0$ and $w_0$ be such that $(x_0,u_0)$ is a singular
point of~\eqref{eq:systpq} and $(x_0,w_0)$ is a singular point of~\eqref{eq:systrs}.
Then $w_0=H(x_0,u_0)$, i.e., the singular points of~\eqref{eq:systpq}
are mapped to singular points of~\eqref{eq:systrs} by the change of variable
$(x,u)\mapsto(x,H(x,u))$. 

In addition, since $u\mapsto H(x_0,u)$ is strictly increasing (by positivity of
the coefficients), it maps the interval $(0,u_0)$  to the
interval $(0,w_0)$. 
\end{lemma}
\begin{proof}
From manipulations of algebraic series (see for
instance~\cite[Sect. 2.9]{GoJa83} for the univariate case and more
recently~\cite[Sect. 7]{ChFuKaSh07} for the bivariate case), one
shows that, when $(x,w)$ and $(x,u)$ are related by $w=H(x,u)$,
$p$ and $q$ are expressed as
\begin{equation}\label{eq:exppq}
p=\frac{(1+r)r}{rs+3s+2s^2+3r+2r^2+1},\ \  q =
\frac{(s+1)s}{rs+3s+2s^2+3r+2r^2+1},
\end{equation}
which implies that
$$
\Delta(p,q)=\tDelta(r,s):=\frac{(1+r)(1+s)(1+r+s-3rs)}{(1+3r+3s+2r^2+2s^2+rs)^2}.
$$

Let $(x_0,u_0)\in\mathbb{R}_+^2$ be a singular point
of~\eqref{eq:systpq} (a solution of $\Delta(p,q)=0$), and let
$w_0=H(x_0,u_0)$. Since $r$ and $s$ are positive at $(x_0,w_0)$,
$\Delta(p,q)=0=\tDelta(r,s)$ implies $1+r+s-3rs=0$.
Since $u\mapsto H(x_0,u)$ is strictly increasing by positivity of the coefficients,
any $w<w_0$ corresponds to some $u<u_0$. Since $(x_0,u_0)$ is a singular point of~\eqref{eq:systpq},
$\Delta(p(x_0,u),q(x_0,u))>0$. So $\tDelta(r(x_0,w),s(x_0,w))>0$. 
 Hence 
$(x_0,w_0)$ is a singular point of~\eqref{eq:systrs}.
\end{proof}

Next we get an expression for $\Sg\ \!\!\!'(x,w)$. Substituting
expression~\eqref{eq:expMgr} of $\Mgr(x,u)$ into~\eqref{eq:MrS}
and replacing $p$ and $q$ by their expressions~\eqref{eq:exppq}, 
we obtain for $\Sg\ \!\!\!'(x,w)$ a rational expression
in terms of $(r,s)$, of the form
\begin{equation}\label{eq:expSgp}
\Sg\ \!\!\!'(x,w)=\frac{1}{\tA(r,s)}\frac{\tPg(r,s)}{(\tDelta(r,s))^{5g-3}},
\end{equation}
where $\tA(r,s)$ ($\tPg(r,s)$, resp.) equals $A(p,q)$ ($P_g(p,q)$, resp.)
after substitution of $p$ and $q$ by their expression~\eqref{eq:exppq}
in terms of $r$ and $s$. (Note that $\tPg(r,s)$ is a rational expression,
 the denominator being a power of $(1+3r+3s+2r^2+2s^2+rs)$.) 

\begin{lemma}\label{lem:Sgp}
For $g\geq 0$, the series $\Sg\ \!\!\!'(x,w)$ admits a singular expansion
of order $\alpha-1$ around every singular  point $(x_0,w_0)$
of~\eqref{eq:systrs}.
\end{lemma}
\begin{proof}
Using again the main theorem in~\cite{drmota97systems}, one shows
that $r$ and $s$ have square-root singular expansions around
$(x_0,w_0)$. Moreover, denoting by $R(x)$ the singularity function
of $r(x,w)$ with respect to $w$, one easily checks, using the
computer algebra system Maple, that $\tDelta(r,s)$ has around
$(x_0,w_0)$ an expansion of the form
$$
\tDelta(r(x,w),s(x,w))=W\cdot h(X,W),\ \mathrm{with}\ 
W=\sqrt{1-w/R(x)},\ X=x-x_0,
$$
and where $h(X,W)$ is analytic at $(0,0)$ and $h(0,0)\neq 0$. Note
also that the expressions of $p$ and $q$ in terms of $r$ and $s$
are positive; and $\tA(r(x_0,w),s(x_0,w))=2uH'(x_0,u)$ (with $w$
and $u$ related by $w=H(x,u)$) is positive for
$w\in(0,w_0)$ and we claim that it converges at $w_0$ (indeed $H(x,u)$
is dominated coefficientwise
 by the series counting rooted planar quadrangulations, satisfying
the ``universal'' singular order $3/2$ at every singular point, 
so $2uH'(x,u)$ is dominated coefficientwise by a series 
of singular order $1/2$, which thus converges at its singular points). 
Hence, looking at
Expression~\eqref{eq:expSgp} of $S_g\ \!\!\!'(x,w)$ and using the 
fact that $\tPg(r(x_0,w_0),s(x_0,w_0))=P_g(p(x_0,u_0),q(x_0,u_0))\neq 0$, 
we conclude that the
radius of convergence of $w\mapsto\Sg\ \!\!\!'(x_0,w)$ must be $w_0$, and
$\Sg\ \!\!\!'(x,w)$ must have a singular expansion of order $(3-5g)/2$.
\end{proof}

Applying Lemma~\ref{lem:integrate} (integration) to the series $S_g\ \!\!\!'(x,w)$, we
obtain the following.
\begin{lemma}\label{lem:Sg}
For $g\geq 0$, the series $\Sg(x,w)$ enumerating near-irreducible
quadrangulations of genus $g$ admits a singular expansion of order
$\alpha$ around every singular point of~\eqref{eq:systrs}.
\end{lemma}

\paragraph{Near-irreducible quadrangulations with fixed edge-width.}
For $g\geq 0$ and $i\geq 1$, let $\cSgei$ be the family of
near-irreducible quadrangulations of genus $g$ with edge-width
$2i$. Define also  $\cSgci$ as the family of quadrangulations from 
$\cSg$ that have a marked non-contractible cycle $C$ of length
$2i$ that is additionally given a direction and carries a marked white vertex, called
\emph{starting vertex} of $C$. 
Denote by $\Sgei(x,w)$ and $\Sgci(x,w)$ the series of
$\cSgei$ and $\cSgci$. Note that $\cSgci$ is a superfamily of
$\cSgei$, so that  $\Sgei(x,w)\preceq\Sgci(x,w)$.


Given $S\in\cSgci$,  let $L=\Phi(S)$ be obtained as follows:
\begin{itemize}
\item
Cut the surface $\mathbb{S}_g$ along $C$ (whether $C$
is surface-separating or not, two cases); this yields 
 two cylindric ends, and a \emph{special face} of
degree $2i$ appears at the boundary of each of the two cylindric
ends; see Figure~\ref{fig:cuts} (orienting $C$ is needed to
distinguish the cylindric ends; there is one on the left of $C$
and one on the right of $C$).

\item
Add a marked black vertex, called a \emph{special vertex},
inside each special face, and connect this vertex
to all white vertices on the contour of the face, see Figure~\ref{fig:cuts}(b).
The special vertex in the left (right, resp.) cylindric end is denoted $v_l$
($v_r$, resp.). Mark the two edges $(v,v_l)$ and $(v,v_r)$, where $v$ is the starting vertex of $C$.
\end{itemize}

\begin{figure}[htb]
\begin{center}
\scalebox{0.38}{\includegraphics{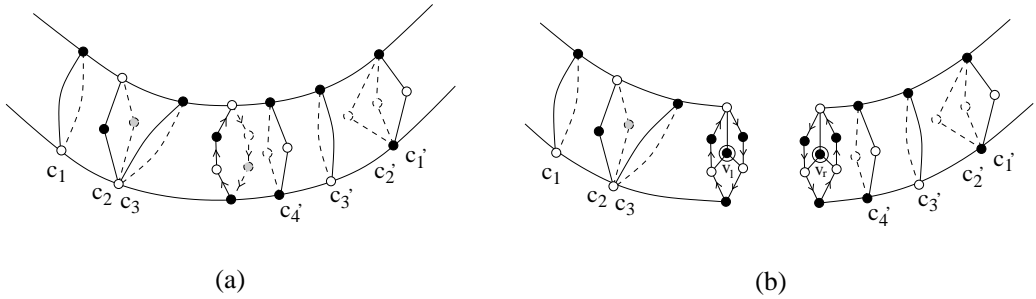}}
\end{center}
\caption{Cutting along a non-contractible cycle cuts a handle
and creates two cylindrical ends.} \label{fig:cuts}
\end{figure}

Let $\cL=\Phi(\cSgci)$ be the family of
well-labelled quadrangulations (i.e., with the $m$ edges relabelled so as to
have distinct labels in $[1..m]$) 
 obtained in this way, and let $L(x,w)$ be the series
of $\cL$. Note that
\begin{equation}\label{eq:SL}
x^{i+2}u^{2i}\cdot S_g^{C=2i}(x,w)= L(x,w).
\end{equation}
An object $L\in\cL$
is of two possible types:
\begin{itemize}
\item
Either $L$ is disconnected (the original cycle $C$ is
surface-separating), in which case $L$ has two connected
components $L_1$ and $L_2$, each component carrying one special
cycle (and one special vertex); in addition, the sum of genera of
$L_1$ and $L_2$ is $g$;  
  \item
or $L$ is connected of genus $g-1$.
 \end{itemize}
We denote by $\cLhs$ the subfamily of objects of $\cL$ of the
first type, and by $\cLhc$ the subfamily of objects of $\cL$ of
the second type.

\begin{lemma}\label{lem:Sgci}
For $g>0$, let as before $\alpha=5(1-g)/2$. Then 
$\Sgci(x,w)$ is bounded coefficientwise by a series  of singular order
$\alpha+1/2$ around every singular  point
$(x_0,w_0)$ of~\eqref{eq:systrs}.
\end{lemma}
\begin{proof}
By Equation~(\ref{eq:SL}), it is enough to prove the result for
the series $L(x,w)$ instead of working with $\Sgci(x,w)$. Let
$S\in\cSgci$, with $L=\Phi(S)$. Let 
$\sigma_S=(c_1,\ldots,c_{\ell},c_r',\ldots,c_1')$ be the possibly
empty sequence of non-contractible 2- or 4-cycles homotopic to
the marked cycle $C$ of $S$, where $c_1,\ldots,c_{\ell}$ are to
the left of $C$ and $c_r',\ldots,c_1'$ are to the right of $C$;
see Figure~\ref{fig:cuts}. 
Here we need to handle with a special care the case of the torus:
for $g=1$, the nested cycles form only one sequence (arranged along a unique cylinder), and by convention we decide that all these cycles are on the left part and form the sequence $(c_1,\ldots,c_{\ell})$, while the right-sequence $(c_r',\ldots,c_1')$ is empty.

When cutting along $C$ (cutting a
handle), $\sigma_S$ is cut into two (possibly empty) sequences
$\sigma=(c_1,\ldots,c_{\ell})$ and $\sigma'=(c_1',\ldots,c_r')$ of
contractible 2- or 4-cycles of $L$, one at each cylindric end.
Note that all cycles of $\sigma$ and of $\sigma'$ must enclose the
special vertex at the corresponding cylindric end. Call $\tcSot$
 the family of rooted planar bipartite maps
---the root is an outer edge directed so as to have the outer face
on its right and a black origin--- with a secondary marked vertex
$v$, an outer
face of degree $2$, all inner faces of degree $4$, and such that
each 2-cycle and each non-facial 4-cycle encloses $v$ in its
interior. Define similarly $\tcSof$, but now with the condition
that the outer face is of degree $4$ with $4$ distinct vertices.

Let $\tSot(x,w)$ and $\tSof(x,w)$ be the series enumerating
$\tcSot$ and $\tcSof$, respectively,  according to black inner
vertices and inner faces. The operation of cutting the surface
along the first cycle of $\sigma$ (if $\sigma$ is not empty) and
the first cycle of $\sigma'$ (if $\sigma'$ is not empty) yields
\begin{equation}\label{eq:Lh}
L(x,w)\preceq K(x,w)\cdot(\tSot(x,w)+\tSof(x,w)+1)^2,
 \end{equation}
 where
 $$
K(x,w):=S_{g-1}\ \!\!\!''(x,w)+\sum_{\substack{g_1+g_2=g\\ g_1>0,g_2>0}}
S_{g_1}\ \!\!\!'(x,w)S_{g_2}\ \!\!\!'(x,w).
 $$
Indeed, each of the two sequences $\sigma,\sigma'$, if not empty,
yields a piece either in $\tcSot$ (say, if $\sigma$ starts with a
2-cycle) or in $\tcSof$ (if $\sigma$ starts with a 4-cycle) at the
corresponding cylindric end. Now let us obtain a singular
expansion for $\tSof(x,w)$ (the method for $\tSot(x,w)$ is
similar)\footnote{In the univariate case, Bender et al.~\cite{BeGaRiWo96} count $\tcSot$
and $\tcSof$ by characterizing the pieces between two successive
cycles of $\sigma$. Instead, our approach to count these classes
relies on a composition scheme.}:

\medskip

\noindent{\bf Claim A.} The series $\tSof(x,w)$ admits a 
singular expansion of positive half-integer order around every
singular  point $(x_0,w_0)$ of~\eqref{eq:systrs}.

\medskip

\noindent\emph{Proof of the claim.} Let $\tcRof$ be the family of rooted 
quadrangulations defined as $\tcSof$, except that any type  of
non-facial 4-cycle is allowed; and let $\tcQof$ be the family of
rooted quadrangulations with 4 distinct outer vertices and a
secondary marked vertex $v_0$. One extracts from $Q\in\tcQof$ a
core $R\in\tcRof$ by collapsing all 2-cycles not strictly enclosing $v_0$.
Hence, similarly as in Equation~\eqref{eq:QR}, 
the series  $\tQof(x,v)$ and $\tRof(x,u)$ enumerating
$\tcRof$ and $\tcQof$ by black inner vertices and inner faces are related by
$$
\tQof(x,u)=\frac{V(x,u)}{u}\tRof(x,V(x,u)),
$$
where $V(x,u)$ is defined in~\eqref{eq:QR}. Similarly one extracts
from $R\in\tcRof$ a core $S\in\tcSof$ by emptying all 4-cycles that are not
strictly enclosing $v_0$. Hence, similarly as in Equation~\eqref{eq:RS}, the series  $\tRof(x,v)$ and
$\tSof(x,w)$ are related by 
$$
\tRof(x,v)=\tSof(x,W(x,v)),
$$
where $W(x,v)$ is defined in~\eqref{eq:RS}. Hence we obtain, with
$H(x,u)=W(x,V(x,u))$,
\begin{equation}
\frac{u}{V(x,u)}\,\tQof(x,u)=\tSof(x,w),\ \ \mathrm{where}\
w=H(x,u).
\end{equation}
One easily shows (from elementary manipulations similar as in~\cite{FuSLC}) 
that $\tQof(x,u)$ and
$V(x,u)$ are expressed rationally in terms of the series $p$ and
$q$ defined in~\eqref{eq:systpq}, and that the singular order of
$\tQof(x,u)$ is $1/2$ at each singular point, whereas the singular
order of $V(x,u)$ is $3/2$ at each singular point. The reason is
that the series of rooted quadrangulations with four distinct
outer vertices is of singular order $3/2$, as any series counting
a ``natural'' family of rooted maps;  so $\tQof(x,u)$ is of singular
order $1/2$ by derivation effect. And $V(x,u)$ is equal to 
$2u(1+Q_0'(x,u))^2/x$, so has the same singular order as
$Q_0'(x,u)$, i.e., $3/2$.

Replacing $p$ and $q$ by their expressions~\eqref{eq:exppq}, one
gets a rational expression for $\tSof(x,w)$ in terms of $r$ and
$s$. As already discussed in Section~\ref{sec:nearirr}, the
singular points of~\eqref{eq:systpq} are mapped by $H(x,u)$ into
the singular points of~\eqref{eq:systrs}. Since $\tQof(x,u)$ and
$V(x,u)$ converge to positive constants at every singular point $(x_0,u_0)$ 
of~\eqref{eq:systpq}, $\tSof(x,w)$  converges at every singular 
point $(x_0,w_0)$ of~\eqref{eq:systrs}. Since $\tSof(x,w)$ is
rationally expressed in terms of $r$ and $s$, which have
square-root singular expansion around $(x_0,w_0)$, we conclude
that $\tSof(x,w)$ admits a singular expansion of order $(2k+1)/2$
around $(x_0,w_0)$, with $k$ a non-negative integer. 
\qedclaim

Similarly, one proves that for every singular  point $(x_0,w_0)$
of~\eqref{eq:systrs}, $\tSot(x,w)$ has a singular
expansion of positive  half-integer order around every singular
point of~\eqref{eq:systrs}.

\medskip
\noindent{\bf Claim B.} The series $K(x,w)$ admits a 
singular expansion of  order $\alpha+1/2$ around every
singular  point $(x_0,w_0)$ of~\eqref{eq:systrs}.

\medskip

\noindent\emph{Proof of the claim.} Let $h=g-1$. Since $S_h(x,w)$ has a
singular expansion of order $\tfrac{5}{2}(1-h)$, by
Lemma~\ref{lem:integrate} $S_h\ \!\!\!''(x,w)$ has a singular expansion of
order  $\tfrac{5}{2}(1-h)-2=\alpha+1/2$ around $(x_0,w_0)$.
Moreover  for each pair $g_1,g_2$ of positive integers adding up to
$g$, the series $S_{g_1}\ \!\!\!'(x,w)$  and $S_{g_2}\ \!\!\!'(x,w)$ have singular
expansions of respective orders $5(1-g_1)/2-1$  and $5(1-g_2)/2-1$
around $(x_0,w_0)$. Hence, by Lemma~\ref{lem:sing_prod}, $S_{g_1}\ \!\!\!'(x,w)S_{g_2}\ \!\!\!'(x,w)$ has a singular expansion of order $3-5g/2=\alpha+1/2$. 
\qedclaim

We now conclude the proof of Lemma~\ref{lem:Sgci}. Since
$K(x,w)$ has a singular expansion of order
$\alpha+1/2$ and $\tSot(x,w)+\tSof(x,w)$ has a
converging singular expansion around $(x_0,w_0)$, then by Lemma~\ref{lem:sing_prod} the product of
these two series (which 
dominates coefficientwise $L(x,w)$ according to~\eqref{eq:Lh}) admits a
singular expansion of order $\alpha+1/2$.
 \end{proof}

\section{Counting graphs of genus $g$ asymptotically}\label{sec:graphs}


The \emph{genus} of a graph $G$ is the minimum genus of a surface
in which $G$ can be embedded. So planar graphs are graphs of genus
$0$.  The \emph{face-width} of a graph $G$ of genus $g\geq 0$ is
the maximum face-width over all embeddings of $G$ in
$\mathbb{S}_g$. Note that the face-width is $+\infty$ in genus $0$
since all curves in the sphere are contractible.

\subsection{Notations for graph families and counting series}
Graph families of connectivity $1$, $2$, $3$ are denoted by the
letters $\cC$, $\cB$, and $\cT$, respectively. The genus $g$ is
indicated as a subscript, and the face-width is indicated as a
superscript. For instance $\cC_g^{\fw\geq 2}$ is the family of
connected graphs of genus $g$ with face-width at least $2$. Unless
explicitly mentioned, all generating functions count graph
families according to vertices and edges, marked respectively by
variables $x$ and $y$. Since vertices are labelled, these series
are exponential in the first variable. For instance, if $c_{n,k}$
is the number of connected graphs of genus $g$ with $n$ vertices
and $k$ edges, then
$$
    C_g(x,y) = \sum c_{n,k} \, y^k {x^n \over n!}.
$$
In the rest of the paper (as opposed to the previous section) 
the derivative of a series $G(x,y)$ according to the \emph{first} variable is denoted
by $G'(x,y)$. 

Given a graph family $\cG$, the \emph{derived} class is the class $\cG'$
of graphs from $\cG$ with one marked vertex that is unlabelled
 and all the other $n$ vertices being as usual labelled by $[1..n]$. 
Similarly the \emph{doubly derived} class $\cG''$ is the class of graphs from
$\cG$ with \emph{two} marked vertices that are unlabelled (distinguished
as a first and a second marked vertex), all the other $n$ vertices being labelled
in $[1..n]$. The counting series of $\cG'$ and $\cG''$ are exponential
in the number of \emph{labelled vertices}, that is, each graph with $n$ labelled
vertices and $k$ edges has weight $x^ny^k/n!$ in the counting series.
If $G(x,y)$ is the counting series of $\cG$, 
it is well known that the counting series of $\cG'$ is $G'(x,y)$ 
and the counting series of $\cG''$ is $G''(x,y)$.

\subsection{Tutte's decomposition of graphs}
We recall here a well-known decomposition of a graph into
3-connected components, as formalized by Tutte~\cite{Tutte}. For
$k\geq 1$, a \emph{$k$-connected graph} is a graph of order at
least $k+1$ in which  at least $k$ vertices need to be deleted to
disconnect it. A connected graph is decomposed  into 2-connected
components, also called blocks, that are articulated around cut
vertices, so that the incidences between the 2-connected
components and the separating vertices form a tree. Similarly, a
2-connected graph $G$ is decomposed at its separating pairs of
vertices into a tree with three kinds of nodes: polygons, multiple
edges, and 3-connected graphs with at least four vertices. The
latter type are called the 3-connected components of $G$. See
\cite{survey} for a detailed discussion in the context of graph
enumeration.

Robertson and Vitray proved the following result.

\begin{theorem}[Robertson and Vitray~\cite{RoVi90}]\label{theo:Rob}
A connected graph $G$ of genus $g>0$ has face-width $k\geq 2$ if
and only if  $G$ has a unique 2-connected component of genus $g$
and face-width $k$, and all the other 2-connected components are
planar.

A 2-connected graph $G$ of genus $g>0$ has face-width $k\geq 3$ if
and only if  $G$ has a unique 3-connected component of genus $g$
and face-width $k$, and all the other 3-connected components are
planar.
\end{theorem}

Let $F(x,y)=xC'(x,y)$ be the series counting vertex-pointed
connected planar graphs (the pointed vertex being labelled). 
Then the first part of
Theorem~\ref{theo:Rob} yields, for $g>0$,
\begin{equation}\label{eq:conn2conn}
C_g^{\fw\geq 2}(x,y)=B_g^{\fw\geq 2}(F(x,y),y).
\end{equation}

A \emph{network}, as defined in~\cite{Wa82a}, is a connected
planar graph with two special vertices, called the poles, 
such that adding an edge connecting the poles gives a
2-connected planar graph. Let $D(x,y)$ be the series counting
networks (the poles are not labelled and not counted by the variable $x$). Then the
second part of Theorem~\ref{theo:Rob} yields, for $g>0$,
\begin{equation}\label{eq:2conn3conn}
B_g^{\fw\geq 3}(x,y)=T_g^{\fw\geq 3}(x,D(x,y)).
\end{equation}

\subsection{Singular expansions}

Again, to genus $g\geq 0$ we associate the quantity
$$
\alpha=\tfrac{5}{2}(1-g).
$$
Our goal is to show that the series enumerating 3-, 2-, and
1-connected graphs of genus $g$ are of approximate-singular order
$\alpha$. 

\subsubsection{3-connected graphs of genus $g$}
We start with 3-connected graphs and use the
result for maps from the previous section. In particular, we use
the fact \cite{RoVi90} (see \cite{mohar92} for an alternative
proof) that if a 3-connected graph of genus $g$ has face-width at
least $2k+3$, then it has a unique embedding in the surface
$\mathbb{S}_g$.

\begin{lemma}[3-connected graphs of genus $g$ and $\fw\geq 3$]\label{lem:Tg}
For $g\geq 0$,  the series $T_g^{\fw\geq 3}(x,w)$ is
approximate-singular of order $\alpha$ around every singular point $(x_0,w_0)$
of~\eqref{eq:systrs}.
\end{lemma}
\begin{proof}
The proof relies again on results by Robertson and
Vitray~\cite{RoVi90}. We use the letter $E$ to denote the series
of embedded 3-connected graphs, that is, 3-connected maps. In this
case we label maps at vertices and use variable $w$ for marking
edges, so that
$$
    E(x,w) = \sum E_{n,m} w^m {x^n \over n!}.
$$
Note that a 3-connected graph yields at least two embeddings (an
embedding and its reflection), hence
$$
2\ \!T_g(x,w)\preceq E_g(x,w).
$$

It is shown in~\cite{RoVi90} that the quadrangulation of a
3-connected map is near-irreducible, hence
$$
E_g(x,w)\preceq S_g(x,w).
$$
There is a technical point here that we clarify. In the previous
section, near-irreducible quadrangulations (counted by the series $S_g$) were
labelled at edges, which correspond to corners of the 3-connected maps. 
But here we are labelling 3-connected maps at vertices, which 
correspond to black vertices in quadrangulations. But in near-irreducible
quadrangulations labelling at edges is equivalent to labelling at black
vertices, in the following sense. If
$S_{n,m}$ is the number of near-irreducible 
quadrangulations with $n$ black vertices and
$m$ faces, labelled at edges with $\{1,2,\dots,2m\}$, and
$\widetilde{S}_{n,m}$ is the number of near-irreducible quadrangulations labelled
at black vertices with $\{1,2,\dots,n\}$, then we have
$$
    S_{n,m} n! = \widetilde{S}_{n,m} (2m)!,
$$
since the previous quantity is the number of quadrangulations
labelled both at vertices and edges. The important point is that
labelling near-irreducible quadrangulations at either black vertices or edges is enough
to avoid symmetries. In order to avoid unnecessary complications,
we use the notation $S_g$ as in Section~\ref{sec:maps}, but taking
into account that $S_g = \sum S_{n,m}x^n w^m/(2m)! = \sum
\widetilde{S}_{n,m} w^m x^n/n!$.

It is also shown in~\cite{RoVi90} that a map is 3-connected of
face-width $k\geq 3$ if and only if  the associated
quadrangulation is irreducible of edge-width $2k$. Hence, for
$k\geq 3$, we have
$$
E_g^{\fw\geq k}=S_g-\sum_{i<k}S_g^{\ew=2i}\succeq S_g-\sum_{i<k}S_g^{C=2i}.
$$
Finally it is shown in~\cite{RoVi90} that if the face-width of a 3-connected graph is at least 
$2g+3$, then the embedding is unique up to reflection, hence
$$
2\ \!T_g^{\fw\geq 2g+3}=E_g^{\fw\geq 2g+3}=S_g-\sum_{i<2g+3}S_g^{C=2i}.
$$
Since
$$
T_g^{\fw\geq 2g+3}\preceq T_g^{\fw\geq 3}\preceq T_g,
$$
we obtain
\begin{equation}
S_g-\sum_{i<2g+3}S_g^{C=2i}\preceq 2\ \!T_g^{\fw\geq 3}\preceq S_g.
\end{equation}
By Lemma~\ref{lem:Sg}, $S_g$ has a singular expansion of order
$\alpha$ around every singular point $(x_0,w_0)$ of~\eqref{eq:systrs}; 
and by Lemma~\ref{lem:Sgci}, for
$i\geq 1$, $\Sgci$ is bounded coefficient-wise by a series $P_g$
that admits a singular expansion of order $\alpha+1/2$ around
$(x_0,w_0)$.  Hence $T_g^{\fw\geq 3}$ is of
approximate-singular order $\alpha$ at $(x_0,w_0)$.
\end{proof}

\subsubsection{2-connected graphs of genus $g$}

\begin{lemma}[2-connected graphs of genus $g$ and $\fw\geq 3$]\label{lem:Bgtt}
For $g\geq 0$, the series $B_g^{\fw\geq 3}(x,y)$  is of strong
approximate-singular order $\alpha$ around every singular  point
$(x_0,y_0)$ of the series $D(x,y)$ of planar networks.
\end{lemma}
\begin{proof}
By Equation~(\ref{eq:2conn3conn}) we have $B_g^{\fw\geq
3}(x,y)=T_g^{\fw\geq 3}(x,D(x,y))$. Bender et al.~\cite{BeGa} have
shown that, if $(x_0,y_0)$ is a singular point of $D(x,y)$, then
$(x_0,w_0)$
---where $w_0=D(x_0,y_0)$--- is a singular point
of~\eqref{eq:systrs}, i.e., is a singular point of $T_g^{\fw\geq
3}(x,w)$. Hence the composition scheme is critical. In addition,
as proved in~\cite{BeGa}, $D(x,y)$ admits a log-free singular expansion of
order $3/2$ around $(x_0,y_0)$. Hence by
Corollary~\ref{coro:comp}, $B_g^{\fw\geq 3}(x,y)$ is of strong 
approximate-singular order $\alpha$ around every singular  point
$(x_0,y_0)$ of $D(x,y)$.
\end{proof}

\begin{lemma}[2-connected graphs of  genus $g$ and  $\fw\geq 2$]\label{lem:Bgt}
For every singular  pair $(x_0,y_0)$ of $D(x,y)$, the series
$B_g^{\fw= 2}(x,y)$ is bounded coefficient-wise by a series
of strong singular order $\alpha+1/2$ at $(x_0,y_0)$.

And the series
$B_g^{\fw\geq 2}(x,y)$  is of strong approximate-singular order
$\alpha$ around $(x_0,y_0)$.
\end{lemma}
\begin{proof}
The proof works by induction on $g$. The first assertion is clearly true for $g=0$
 since the face-width is infinite by convention. The second assertion has
been shown by Gim\'enez and Noy for planar graphs in~\cite{gimeneznoy}.

Let $g>0$, and assume the property is true up to $g-1$.
Let $\cE_g$ be the family of well-labelled graphs with two connected components:
the first one in $\cB_g^{\fw= 2}$, the second one an oriented edge
(the second component is only an artifact to have a reserve of two labelled vertices).
Note that the series of $\cE_g$ is
$$
E_g(x,y)=x^2y\ \!\!B_g^{\fw= 2}(x,y)
$$
Let $\cF_g$ be the family of graphs  
with two ordered connected components of positive 
genera adding up to $g$, each connected component being a 2-connected
graph of face-width at least $2$ with additionally a marked directed edge. Note that the series of $\cF_g$ satisfies
$$
F_g(x,y)=\sum_{\substack{g_1+g_2=g\\ g_1>0,g_2>0}} 2y\frac{\partial}{\partial y}B_{g_1}^{\fw\geq 2}(x,y)\cdot 2y\frac{\partial}{\partial y}B_{g_2}^{\fw\geq 2}(x,y).
$$
And let $\cH_g$ be the family of 2-connected graphs of genus at most $g$
and face-width at least $2$ and having two (a first and a second) marked directed 
edges. Note that the series of $\cH_g$ satisfies
$$
H_g(x,y)=\sum_{h\leq g}\frac{\partial^2}{\partial y^2}B_g^{\fw\geq 2}.
$$
We are going to show an injection from $\cE_g$ to $\cF_g+\cH_{g-1}$
(since $\cF_g$ and $\cH_{g-1}$ only involve family of 2-connected graphs
of smaller genus, this will allow us to conclude the proof by induction).

Let $(G,e)\in\cE_g$. Let $M$ be an embedding of $G$ of face-width $2$
and let $C$ be a non-contractible cycle intersecting $M$ at two vertices
$v$ and $v'$, $C$ being additionally directed. Cut the surface along $C$
as in Figure~\ref{fig:cuts}, making two cylindric ends appear (one on the left of $C$,
on one the right of $C$). In the tip of each cylindric end, add an edge directed 
from $v$ to $v'$. So $v$ ($v'$, resp.) is split into two vertices: 
the one on the left cylindric end keeps the label of $v$ (of $v'$, resp.), the one on the 
right cylindric end receives the label of the origin (end, resp.) of the edge-component
$e$~\footnote{The edge-component 
is only used to make the resulting graph well-labelled.}.
Let $M'$ be the resulting
embedded graph and $G'$ the underlying graph.
Note that if $C$ is surface-separating $G'$ has two ordered connected components
(the first one contains the left cylindric end) 
of positive genera adding up to $g$, 
each connected component having additionally a marked directed edge.
If $C$ is not surface-separating, $G'$ has one connected components
of genus at most $g-1$ and has additionally two marked directed edges.
Clearly the mapping from $(G,e)$ to $G'$ is injective, since $G$ is recovered
from $G'$ by identifying the two marked directed edges: the merged edge
receives the labels of the first marked edge, and $e$ receives the labels of the
second marked edge.

We now claim that the image of $\cE_g$ is included in $\cF_g+\cH_{g-1}$,
i.e., that the connected components of the graph associated with $G\in\cE_g$
are 2-connected of face-width at least $2$. This is best seen by going 
to quadrangulations. Take the case where $C$ is not surface-separating
(the other case works similarly). With the notations above, let $Q$
be the quadrangulation associated with $M$, and $Q'$ the quadrangulation
associated with $M'$. The effect of cutting along $C$
is the same as cutting along a 4-cycle $c$ of $Q$, so $Q'$ is obtained
from $Q$ by cutting along $c$. 
Recall that simple bipartite quadrangulations of genus $g$ exactly correspond
to 2-connected maps of face-width at least $2$. 
Since $M$ is 2-connected of face-width at least
$2$, $Q$ is a simple quadrangulation, so $Q'$ is also a simple quadrangulation,
so $M'$ is a 2-connected map of face-width at least $2$ (and genus $g-1$).
Since $M'$ has face-width at least $2$ and genus $g-1$, $G'$ has face-width
at least $2$ and genus at most $g-1$.

Thanks to the injective mapping we have constructed it is an easy matter to conclude the proof by induction. 
By injectivity we have
$$
x^2yB_g^{\fw= 2}(x,y)=E_g(x,y)\preceq F_g(x,y)+H_{g-1}(x,y).
$$
By induction, each series $B_h^{\fw\geq 2}(x,y)$ for $h<g$ is of strong approximate-singular order
$5/2(1-h)$ at every singular point $(x_0,y_0)$ of $D(x,y)$. Hence the partial
derivative according to $y$ of such a series is of 
order $3/2-5/2h$. Hence, for $g_1+g_2=g$, the product of such two series
is of order $3-5/2g$, so $F_g(x,y)$ is also of order $3-5/2g=\alpha+1/2$.
Similarly, $H_{g-1}(x,y)$ is of order $\alpha+1/2$, so the sum of $F_g$ and $H_{g-1}$,
which dominates $B_g^{\fw= 2}$, is of order $\alpha+1/2$. 
Finally, the series $B_g^{\fw\geq 2}=B_g^{\fw\geq 3}+B_g^{\fw=2}$ is the sum
of two series respectively of strong approximate-singular orders $\alpha$ (according to Lemma~\ref{lem:Bgtt}) and $\alpha+1/2$ (as just shown), so the second term
is negligible and  $B_g^{\fw\geq 2}$ is thus of
strong approximate-singular order $\alpha$. 
\end{proof}

Even if only Lemma~\ref{lem:Bgt} is needed to do the asymptotics of connected graphs
of genus $g$, we add here for the sake of completeness a lemma (bounding the series of
2-connected graphs of face-width~$1$) in order to have a theorem on the 
asymptotic enumeration of 2-connected graphs of genus $g$ (Theorem~\ref{theo:2conn}
at the end of this section). 

\begin{lemma}[2-connected graphs of  genus $g$]\label{lem:Bg}
For every singular  pair $(x_0,y_0)$ of $D(x,y)$, the series
$B_g^{\fw= 1}(x,y)$ is bounded coefficient-wise by a series
of strong singular order $\alpha+1/2$ at $(x_0,y_0)$.

And the series
$B_g^{\fw\geq 1}(x,y)=B_g(x,y)$  is of strong approximate-singular order
$\alpha$ around $(x_0,y_0)$.
\end{lemma}
\begin{proof}
Similarly as in Lemma~\ref{lem:Bgt} we reason by induction on $g$.
The stated assertions are true for $g=0$ (the first assertion is trivial, the second
assertion is proved in~\cite{gimeneznoy}).

Let $g>0$, and assume the stated properties are true up to genus $g-1$.
Given a 2-connected graph $G$ of genus $g$ and face-width $1$, let $M$
be an embedding of $G$ on $\mathbb{S}_g$ with face-width $1$,
and let $C$ be a non-contractible 
closed curve intersecting $M$ at a single vertex, call it $v$.
Cutting along $C$ splits $v$ into two ``marked'' vertices and yields 
a connected graph $M'$ 
embedded on  $\mathbb{S}_{g-1}$ and with two marked vertices (the case
of $C$ surface-separating is excluded, since it would contradict the 
2-connectivity of $G$).
Let $G'$ be the underlying connected graph, so $G'$ has genus at most $g-1$.
Note that even after forgetting the embedding, $G$ can be recovered from $G'$
by merging the 2 marked vertices of $G'$. So the mapping from $G$ to $G'$ and its two marked vertices
is injective (upon using a similar artifact as in Lemma~\ref{lem:Bgt} to make $G'$
well-labelled).  
The graph $G'$ is not necessarily 2-connected, but close to. We use the following
easy claim:

\medskip

\noindent{\bf Claim A.} Let $G$ be a 2-connected graph, let $v$ be a
vertex of $G$ and $(E_1,E_2)$ a partition of the edges incident to $v$
into 2 non-empty parts. Let $G'$ be the graph obtained from $G$
by splitting $v$ into 2 vertices $v_1$ and $v_2$, $v_i$ keeping as incident
edges the edges in $E_i$ for $i\in\{1,2\}$. Then the block-structure of $G'$ 
is a chain of 2-connected components $B_1,\ldots,B_k$, i.e., each $B_i$ has
two distinct marked vertices $w_i,w_i'$ such that $w_1=v_1$, $w_k=v_2$,
and $w_i'=w_{i+1}$ for $i\in[1..k-1]$.  

\medskip

Hence the graph $G'$ obtained from $G$ is a chain of bi-pointed 2-connected
graphs with sum of genera strictly smaller than $g$ (because the genus
of a graph is the sum of genera of its blocks). For each composition $\kappa$ of 
$k\leq g-1$ as a sum of positive integers $i_1,\ldots,i_j$, let $\cS^{(\kappa)}$
be the family of connected graphs $G'$ made of a chain of blocks,
so that the  blocks of positive genera in the chain have genera $i_1,\ldots,i_j$
and there are two unlabelled marked vertices in $G'$ contained respectively
in the first and in the last block of the chain.
Note that (with $\cZ$ marking a labelled vertex, $\mathrm{Seq}$ denoting
the sequence construction, and $\star$ denoting the labelled product in the sense of~\cite{FlaSe}) 
$$
\cS^{(\kappa)}=\cB_{i_1}\ \!\!\!''\star\cdots\star\cB_{i_j}\ \!\!\!''\star\mathrm{Seq}(\cZ\star\cB_0\ \!\!\!'')^{j+1}.
$$

\medskip

\noindent{\bf Claim B.} Let $\kappa=i_1+\cdots+i_j$
be a composition of $k<g$ by positive integers.
Then the generating series $S^{(\kappa)}(x,y)$ of $\cS^{(\kappa)}$
is of strong approximate-singular order $j/2-5/2k$ at every singular
point $(x_0,y_0)$ of $D(x,y)$.  

\medskip

\noindent\emph{Proof of the claim.} By induction,  for $h<g$, 
each series $B_h\ \!\!\!''(x,y)$ is of strong approximate-singular order
$1/2-5/2h$ at $(x_0,y_0)$. Hence, by Lemma~\ref{lem:sing_prod},
 $B_{i_1}\ \!\!\!''(x,y)\cdots B_{i_j}\ \!\!\!''(x,y)$
is of strong approximate-singular order $j/2-5/2k$. 
To conclude the proof of the claim we just have 
to show that the series $S(x,y)$ of $\mathrm{Seq}(\cZ\star\cB_0)''$ is of (strong) 
positive approximate-singular order; we show it is of order $1/2$. 
Since $B_0(x,y)$ is of order $5/2$ at $(x_0,y_0)$, $B_0''(x,y)$ is of order
$1/2$, so $1/(1-xB_0''(x,y))$ is either of order $1/2$ (if $x_0B_0''(x_0,y_0)=1$)
or is of order $-1/2$ (if $x_0B_0''(x_0,y_0)<1$), so proving that $S(x,y)$
is of order $1/2$ reduces to proving that $1/(1-x_0B_0''(x_0,y_0))$ is finite. 
Recall the equation 
$$
F(z,y)=z\exp(B_0\ \!\!\!'(F(z,y),y))
$$
relating the series $F(z,y)$ of pointed connected planar graphs and the series
$B_0(x,y)$ of 2-connected planar graphs. Differentiating this equation with respect to $z$ yields
$$
F'(z,y)=F(z,y)F'(z,y)B_0\ \!\!\!''(F(z,y),y)+F(z,y)/z,
$$
so
$$
F'(z,y)=\frac{x}{z}\frac1{1-xB_0\ \!\!\!''(x,y)},\ \ \mathrm{where}\ x=F(z,y).
$$
Let $z_0$ be the positive value such that 
$(z_0,y_0)$ is a singular point of $F(z,y)$. As proved in~\cite{gimeneznoy}, 
the composition scheme 
from 2-connected to connected planar graphs is critical, i.e., $x_0=F(z_0,y_0)$. 
Moreover, it is shown in~\cite{gimeneznoy} that
$F(z,y)$ is of singular order $3/2$ at $(x_0,y_0)$,
so $F'(z,y)$ is of order $1/2$. Hence $F'(z,y)$ converges at $(z_0,y_0)$,
which implies that $1/(1-xB_0\ \!\!\!''(x,y))$ converges at $(x_0,y_0)$, and thus
is of singular order $1/2$. Observe that the same argument shows that the denominator 
$1-xB_0\ \!\!\!''(x,y)$ does not vanish before we reach the singular point $(x_0,y_0)$. This concludes the proof of the claim.
 \qedclaim

The injection presented in the beginning of the proof of the lemma ensures that
$B_g^{\fw=1}(x,y)$ is bounded coefficientwise by the sum of series $B^{(\kappa)}(x,y)$,
where $\kappa$ runs over compositions of integers strictly smaller than $g$.
By Claim~B, 
each series $B^{(\kappa)}(x,y)$ is of strong approximate-singular order at least
$1/2-5/2(g-1)=\alpha+1/2$ (the least singular order of $B^{(\kappa)}(x,y)$
is attained for $\kappa$ the composition of $g-1$ with a single part). 
This completes the proof of the first assertion of the lemma.

The second assertion follows straightforwardly. The series
$B_g^{\fw\geq 1}(x,y)=B_g(x,y)$ is the sum of $B_g^{\fw\geq 2}(x,y)$
and $B_g^{\fw=1}$, the first term being of strong
approximate-singular order $\alpha$ by Lemma~\ref{lem:Bgt}, 
and the second term being
bounded coefficientwise by a series of strong approximate-singular order
$\alpha+1/2$, hence of negligible contribution. Hence $B_g(x,y)$
is also of strong approximate-singular order $\alpha$ at $(x_0,y_0)$. 
\end{proof}

Applying the transfer theorems stated in Section~\ref{sec:sing} to the 
series counting 2-connected graphs of genus $g$, we obtain:
 
\begin{theorem}\label{theo:2conn}
For $g\geq 0$, the number $b_n^{(g)}$ of 2-connected graphs with $n$ vertices
that can be embedded in the orientable surface  $\mathbb{S}_g$ of
genus $g$ satisfies
\begin{equation}\label{eq:asympt-2conn}
b_n^{(g)} \sim \,d^{(g)}n^{5(g-1)/2-1}\gamma^n n!
\end{equation}
where $d^{(g)}$ is a positive constant and $\gamma$ is the exponential
growth constant of 2-connected planar graphs.

For $\mu\in(1,3)$, the number $b_{n,m}^{(g)}$ of graphs with $n$
vertices and $m=\lfloor \mu n \rfloor$ edges that can be embedded in
$\mathbb{S}_g$ satisfies
$$
b_{n,m}^{(g)} \sim \, d_{\mu}^{(g)}n^{5g/2-4}(\gamma_{\mu})^n n!\ \
when \ n\to\infty,
$$
where $d_{\mu}^{(g)}$ is a positive constant and $\gamma_{\mu}$ is
the exponential growth constant of planar graphs with ratio edges/vertices
tending to $\mu$ ($\gamma$ and 
$\gamma(\mu)$ are characterized analytically in~\cite{gimeneznoy}).
\end{theorem}

\subsubsection{Connected graphs of genus $g$}

\begin{lemma}[Connected graphs of genus $g$ and $\fw\geq 2$]\label{lem:Cgfwt}
For $g\geq 0$, the series $C_g^{\fw\geq 2}(x,y)$  is of strong 
approximate-singular order $\alpha$ around every singular  pair
$(x_0,y_0)$ of the series $G_0(x,y)$ counting planar
graphs.
\end{lemma}
\begin{proof}
In this proof the first variable of the series of 2-connected graph families is denoted
by $z$ to avoid confusion. Recall Equation~(\ref{eq:conn2conn}): 
$C_g^{\fw\geq 2}(x,y)=B_g^{\fw\geq 2}(F(x,y),y)$. Gim\'{e}nez and
Noy~\cite{gimeneznoy} have shown that, if $(x_0,y_0)$ is a
singular point of $G_0(x,y)$, then $(z_0,y_0)$ ---where
$z_0=F(x_0,y_0)$--- is a singular point of $D(z,y)$. Hence,
according to Lemma~\ref{lem:Bgt}, $(z_0,y_0)$ is also a singular
point of $B_g^{\fw\geq 2}(z,y)$ (precisely, is a singular point of
the two series bounding $B_g^{\fw\geq 2}(z,y)$ from above and
below). Hence the composition scheme is critical. In addition, as
proved in~\cite{gimeneznoy}, $F(x,y)$ admits a log-free singular expansion of
order $3/2$ around $(x_0,y_0)$. Hence by
Corollary~\ref{coro:comp}, $C_g^{\fw\geq 2}(x,y)$ is of strong
approximate-singular order $\alpha$ around every singular  point
$(x_0,y_0)$ of $F(x,y)$.
\end{proof}

\begin{lemma}[Connected graphs of genus $g$, and with fixed face-width]
	\label{lem:Cg}
For $g\geq 0$ and $k\geq 0$, the series $C_g^{\fw=k}(x,y)$  is bounded
coefficientwise by a series that is of strong 
approximate-singular order $\alpha+1/2$ around every singular
point $(x_0,y_0)$ of the series $G_0(x,y)$ counting planar
graphs. 

And the series $C_g(x,y)$  is of strong
approximate-singular order $\alpha$ around 
$(x_0,y_0)$.
\end{lemma}
\begin{proof}
As in Lemma~\ref{lem:Bgt} (and using similar notations), 
the proof works by induction on $g$. The stated properties are true
when $g=0$ since by convention the face-width of any planar graph is infinite. 
Assume the property is true up to $g-1$, for $g>0$. 
Let $\cE_g^{(k)}$ 
be the family of graphs (as usual well-labelled, the vertices
have distinct labels in $[1..n]$) made of two connected components,
the first one in $\cC_g^{\fw=k}$, the second
one an oriented path of $k$ vertices (again the second component is just
an artifact to have a reserve of $k$ labelled vertices). 
Note that the series of $\cE_g^{(k)}$ is $x^kC_g^{\fw=k}(x,y)$. 
Let $\cF_g$ be the family of graphs with two connected components
of positive genera $(g_1,g_2)$ adding up to $g$,
each connected component having a marked vertex that is unlabelled;
so the series of $\cF_g$ satisfies
$$
F_g(x,y)=\sum_{\substack{g_1+g_2=g\\ g_1>0,g_2>0}}C_{g_1}\ \!\!\!'(x,y)C_{g_2}\ \!\!\!'(x,y).
$$
And let $\cH_g$ be the doubly derived family of connected graphs of genus at most $g$
 (i.e., with two marked unlabelled vertices): $\cH_g=\cup_{h\leq g}\cC_h\ \!\!\!''$,
 so the series of $\cH_g$ satisfies
 $$
 H_g(x,y)=\sum_{h\leq g}C_h\ \!\!\!''(x,y).
 $$ 
Denote by $\frak{S}_k$ the group of permutations of $k$ elements. 
We are going to define an injective mapping from $\cE_g^{(k)}$
to $(\cF_g+\cH_{g-1})\times\frak{S}_k\ \!\!\!^2$. (Note that $\cF_g$ and $\cH_{g-1}$
only involve families of connected graphs of genus smaller than $g$, which will allow us 
to use induction to conclude the proof).

Let $(G,p)\in\cE_g^k$. 
Associate with  $G$  an embedding
$M$ (map of genus $g$) of face-width $k$.
 Let $C$ be an oriented non-contractible curve intersecting $M$ at $k$ vertices,
 one of which is marked as the ``starting vertex'' of $C$.
Cut $M$ along $C$, thus creating 
two cylindric ends, as in Figure~\ref{fig:cuts}. When doing this, each vertex along $C$ is
split into two vertices, one in each cylindric end; to make the resulting graph
well-labelled, the vertices $(v_1,\ldots,v_k)$
of the contour of $C$ (with $v_1$ the starting vertex) on the right tip are labelled
according to $p$, i.e., $v_i$ receives the label of the $i$th vertex on the path $p$. 

This process either yields
two maps whose genera add up to $g$ (case where $C$ is surface-separating)
or cuts a handle and yields a map of genus $g-1$.
Add a special vertex $v_{\mathrm{left}}$ ($v_{\mathrm{right}}$, resp.) in the 
 tip-area
of  the left (right, resp.) cylindric end, and connect the special vertex
to the $k$ vertices on the contour of the tip-area (in Figure~\ref{fig:cuts}
a similar operation occurs, but the special vertex is only connected
to the white vertices of the contour). 
Let $G'$ be the graph (with two components if $C$ is surface-separating,
with one component otherwise) obtained after this process and forgetting
the embedding. 
Note that the $k$ neighbours of $v_{\mathrm{left}}$
and the $k$ neighbours of $v_{\mathrm{right}}$ were matched in $G$ along $C$; 
record this matching by two permutations $(\sigma,\sigma')\in\frak{S}_k$ defined
as follows:  if $(v_1,\ldots,v_k)$ is the occurrence of $C$ on the left (right, resp.) tip, then
$\sigma(i)=j$  ($\sigma'(i)=j$, resp.) 
means that $v_i$ has the $j$th smallest label among $v_1,\ldots,v_k$.

We claim that the mapping from $(G,p)$ to $(G',\sigma,\sigma')$ is injective.
Indeed $(G,p)$ can be recovered from $(G',\sigma,\sigma')$ as follows: 
\begin{itemize}
\item
For $i\in[1..k]$, let $u_i$ be the neighbour 
of $v_{\mathrm{left}}$ in $G'$ with $\sigma(i)$th smallest label,
and let $u_i'$ be the neighbour of $v_{\mathrm{right}}$ in $G'$ 
with $\sigma'(i)$th smallest label.
Then $G$ is recovered from $G'$ by merging $u_i$ with $u_i'$ for $i\in[1..k]$
(the merged vertex keeps the label of $u_i$) and then by  
removing the special vertices $v_{\mathrm{left}}$ and $v_{\mathrm{right}}$ and
their incident edges.
\item
And $p$ is recovered as the oriented path of $k$ vertices, where for $i\in[1..k]$ 
the $i$th vertex of $p$ 
receives the label of  $u_i'$.
\end{itemize}

 Since the construction is injective, we have
 $$
 x^kC_g^{\fw=k}(x,y)=E_g^{(k)}(x,y)\preceq (F_g(x,y)+H_{g-1}(x,y))\cdot k!^2.
 $$
 By induction, for $h<g$, 
 $C_h(x,y)$ is of strong approximate-singular order $5/2(1-h)$ at any 
singular point $(x_0,y_0)$ of $G_0(x,y)$ (series counting connected planar graphs), 
hence $C_h\ \!\!\!'(x,y)$ ($C_h\ \!\!\!''(x,y)$, resp.) 
 is of strong approximate-singular order $3/2-5/2h$ ($1/2-5/2h$, resp.). 
 Hence $F_g(x,y)$ (by Lemma~\ref{lem:sing_prod}) and $H_g(x,y)$ are both 
 of strong approximate-singular order
 $3-5/2g=\alpha+1/2$ at $(x_0,y_0)$. This proves the first assertion.
 
The second assertion follows easily. 
Indeed $C_g=C_g^{\fw\geq 2}+C_g^{\fw=1}$. 
By Lemma~\ref{lem:Cgfwt}, 
the first term is of strong approximate-singular order $\alpha$. 
 And the second term is negligible ---bounded by a series of strong 
 approximate-singular order 
 $\alpha+1/2$ at $(x_0,y_0)$--- according to the first assertion
applied to $k=1$.  
\end{proof}

\subsubsection{Graphs of genus $g$}

\begin{lemma}[Graphs of genus $g$]\label{lem:Gg}
For $g\geq 0$, the series $G_g(x,y)$ counting graphs of genus $g$
is of strong approximate-singular order $\alpha$ around every
singular  point $(x_0,y_0)$ of the series $G_0(x,y)$ counting
 planar graphs.
\end{lemma}
\begin{proof}
Again the proof works by induction on $g$.
The result has been proved by Gim\'enez and Noy~\cite{gimeneznoy}
for $g=0$.
Let $g>0$, and assume the stated property is true up to genus $g-1$. 
The genera of the connected components of a genus $g$ graph $G$
add up to $g$. If two connected components are non-planar then
there is a partition of $G$ into two non-planar graphs $G_1, G_2$
whose genera add up to $g$
\begin{equation}
G_g(x,y)=C_g(x,y)G_0(x,y)+K_g(x,y),
\end{equation}
where
$$
K_g(x,y)\preceq \sum_{\substack{g_1+g_2=g\\
g_1>0,g_2>0}}G_{g_1}(x,y)G_{g_2}(x,y).
$$
By induction, for $h<g$, $G_{h}(x,y)$ is of strong
approximate-singular order $5(1-h)/2$ around $(x_0,y_0)$. Hence,
for $g_1>0$, $g_2>0$ such that $g_1+g_2=g$, Lemma~\ref{lem:sing_prod} ensures that 
the series
$G_{g_1}(x,y)G_{g_2}(x,y)$ is of strong approximate-singular
order $5/2(1-g_1)+5/2(1-g_2)=\alpha+5/2$ around $(x_0,y_0)$. 
Thus the dominating part of
$G_g(x,y)$ is $C_g(x,y)G_0(x,y)$, which by Lemma~\ref{lem:sing_prod} 
is of strong approximate-singular
order $\alpha$ around $(x_0,y_0)$ (indeed $C_g(x,y)$ is of
approximate-singular order $\alpha$ by Lemma~\ref{lem:Cg} and
$G_0(x,y)$ is of approximate-singular order $5/2$, as proved
in~\cite{gimeneznoy}).
\end{proof}

\begin{lemma}[Graphs embeddable on $\mathbb{S}_g$]\label{lem:Fg}
For $g\geq 0$, the series $F_g$ counting graphs embeddable on the
genus $g$ surface is of strong approximate-singular order
$\alpha$ around every singular  point $(x_0,y_0)$ of the series
$C(x,y)$ counting connected planar graphs.
\end{lemma}
\begin{proof}
Note that a graph embeddable on $\mathbb{S}_g$ has genus at most
$g$, hence
$$
F_g(x,y)=\sum_{i=0}^gG_i(x,y).
$$
By Lemma~\ref{lem:Gg}, each series $G_i(x,y)$ in the sum is
 of strong approximate-singular order $5(1-i)/2$, hence the
dominating series in the sum is $G_g(x,y)$, which is of strong 
approximate-singular order $5(1-g)/2=\alpha$.
\end{proof}

To conclude, applying the transfer theorems
(Corollary~\ref{coro:transfer}) to the singular expansions of the
series counting graphs embeddable on $\mathbb{S}_g$
(Lemma~\ref{lem:Fg}), we obtain
 the asymptotic enumeration results stated in Theorem~\ref{theo:main}.

 \subsection{The non-orientable case}
The proof of Theorem~\ref{theo:main2} follows  the same
lines as the proof of Theorem~\ref{theo:main}.
Here, we explain the mild differences and how to adapt the proof
at those points. 
Let $h>0$ and let $\mathbb{N}_h$ be the non-orientable surface with $h$ crosscaps.
In this section, we use the same notations as before for the generating series of maps and graphs, but with a different meaning: all series will refer to non-orientable surfaces, and the subscript in the notation will refer to the non-orientable genus $h$ (whereas in previous sections it referred to the oriented genus). For example, the generating series of near-irreducible quadrangulations on the surface $\mathbb{N}_h$ will be denoted by $S_h(x,w)$.

Before giving the outline of the proof of Theorem~\ref{theo:main2}, let us give some definitions. The \emph{non-orientable genus} of a graph $G$ is the smallest $h$ such that $G$ can be embedded in the non-orientable surface $\mathbb{N}_h$, and the \emph{non-orientable face-width} of $G$ is the largest face-width among all the embeddings of~$G$ on $\mathbb{N}_h$. Since we will have to consider at the same time orientable and non-orientable surfaces, it makes sense to introduce the \emph{Euler genus} $\kappa$ of a surface $S$, which is defined as $\kappa= 2-\chi(S)$, where $\chi(S)$ is the Euler characteristic of $S$. Therefore $\kappa = 2g$ for the orientable surface $\mathbb{S}_g$, and $\kappa=h$ for the non-orientable surface $\mathbb{N}_h$. The \emph{Euler genus} of a graph is the minimal Euler genus of a surface (orientable or non-orientable) in which $G$ can be embedded.

We now define $\alpha=\tfrac{5}{2}(1-h/2)$ (as we will see, the singular exponent for a given surface depends only on its Euler genus, and $h$ plays the same role here as $2g$ in the orientable case). 
The bijection between maps and bipartite quadrangulations
also works on the non-orientable surface $\mathbb{N}_h$, and the extraction
of a near-irreducible core works in the same way, so the equation
relating rooted maps and near-simple quadrangulations on $\mathbb{N}_h$
is the same as for $\mathbb{S}_g$
(Equation~\eqref{eq:MrS}). The only difference is in the expression of 
the series $\overrightarrow{M}_h(x,u)$ of rooted maps on $\mathbb{N}_h$;
Arqu\`es and Giorgetti have shown in~\cite{ArGi} that there exists a trivariate polynomial
$P_h(X,Y,Z)$ such that
$$
\overrightarrow{M}_h(x,u)=\frac{P_h(p,q,\sqrt{\Delta})}{\Delta^{5h/2-3}},
$$
i.e., the numerator-polynomial has  $\sqrt{\Delta}$ as additional parameter, where $\Delta$ is as before the Jacobian of~(\ref{eq:systpq}).
Recall that around a singular point of~\eqref{eq:systpq}, $\Delta(p(x,u),q(x,u))$
is of the order of $\sqrt{1-x/\rho(u)}$, so $\sqrt{\Delta}$ is of the order
of $(1-x/\rho(u))^{1/4}$. Hence, in the non-orientable case 
singular expansions are of the same form as in the orientable case, but
in terms of $\widetilde{X}=(1-x/\rho(u))^{1/4}$ instead of $X=(1-x/\rho(u))^{1/2}$,
which is a mild extension of the definition given in Section~\ref{sec:sing}.
The definition of the order of a singular expansion is also 
readily extended.

Similarly as in the orientable case, the sketch of proof in~\cite{BeCaRi93} ensures
that $P_h$ is non-zero at a singular point $(x_0,u_0)$ for $(p,q)$, so 
$\overrightarrow{M}_h(x,u)$ admits a singular expansion of order $\alpha$
at $(x_0,u_0)$. 
Adapting Equation~(\ref{eq:expSgp}), 
the expression of the series $S_h\ \!\!\!'$  of rooted 
near-simple quadrangulations on $\mathbb{N}_h$ is of the form
$$
S_h\ \!\!\!'(x,w)=\frac{1}{\tA(r,s)}\frac{\tPg\Bigl(r,s,\sqrt{\tDelta(r,s)}\,\Bigr)}{(\tDelta(r,s))^{5g-3}}.
$$
All the arguments of Lemma~\ref{lem:Sgp} still apply, even with 
the extra parameter $\sqrt{\tDelta}$ in the numerator-polynomial,
so $S_h$ has a singular expansion of order $\alpha$ at every singular point
of~\eqref{eq:systrs}. 
In the statement of Lemma~\ref{lem:Sgci}, 
$\alpha+1/2$ has to be changed to $\alpha+1/4$
(this is no problem, the contribution of near-irreducible 
quadrangulations of a fixed edge-width is still
negligible since the singular order is larger than the one of near-irreducible quadrangulations).
This is due to the fact that there is one more case to consider in the proof:
that the marked cycle $C$ of length $2i$ crosses a crosscap. In that case, cutting along
$C$ removes the crosscap, and yields a surface of Euler genus $h-1$. So we obtain either a quadrangulation on $\mathbb{N}_{h-1}$ or a quadrangulation
on the orientable surface $\mathbb{S}_g$, where $g=\frac{h-1}{2}$ (the latter case may only appear if $h$ is odd).
In both cases, $C$ yields a boundary of length $4i$ delimiting a ``hole'' in the surface.
Filling this hole with a disk and placing a marked black vertex connected
to all white vertices on the contour, we obtain a quadrangulation on $\mathbb{N}_{h-1}$ or $\mathbb{S}_g$
with a marked black vertex. This quadrangulation may not be near-irreducible.
However, any contractible 2-cycle or 4-cycle 
bounds a disk that contains the marked black vertex. These cycles are therefore nested around
the marked vertex. Hence by similar 
arguments as in Lemma~\ref{lem:Sgci},
the contribution of ``$C$ crossing a crosscap" is bounded by a series of the same
singular order as the series counting near-irreducible quadrangulations on $\mathbb{N}_{h-1}$ with a marked vertex, plus those on $\mathbb{S}_g$ (when $h$ is odd).
 By induction, the latter series with no marked
vertex is of singular order $\alpha+5/4$
(since $\alpha+5/4=\frac{5}{2}(1-(h-1)/2)=\frac{5}{2}(1-g)$). The effect of marking a vertex is to decrease the order
by $1$, so the bounding series is actually of order $\alpha+1/4$.
The other two cases ($C$ being surface-separating and $C$ cutting a handle)
still yield contributions bounded by a series of singular order $\alpha+1/2$.
Let us remark that in these cases we may also end up with an orientable surface $\mathbb{S}_{g'}$ but since for each $g'$ the singular orders for $\mathbb{S}_{g'}$ and $\mathbb{N}_{2g'}$ are the same, the orientability of the surface obtained after cutting along $C$ has no effect on the singular order of the corresponding bounding series.

The next step is to go to 3-connected graphs of non-orientable genus $h$ via
3-connected maps on $\mathbb{N}_h$. 
By Proposition 5.5.12 in~\cite{Mohar}, a map on $\mathbb{N}_h$ is 3-connected
if the associated quadrangulation is irreducible (no 2-cycles and no non-facial 4-cycles at all). Conversely, 
the quadrangulation associated with a 3-connected map of face-width $\geq 2$ is near-irreducible (no 2-cycles and no contractible non-facial 4-cycles).
So 3-connected maps on $\mathbb{N}_h$ have the same
approximate-singular expansion as near-irreducible quadrangulations. 
Since a 3-connected graph of non-orientable genus $h$ has a
unique embedding if the face width is at least $2h+3$
\cite{RoVi90}, one shows that the corresponding series is of strong
approximate-singular order $\alpha$ similarly as in Lemma~\ref{lem:Tg}.
Finally, one proceeds with 2-connected graphs of non-orientable genus $h$
and then with connected graphs and arbitrary graphs
similarly as in the orientable case, since the statements in Theorem~\ref{theo:Rob}
also hold with the non-orientable genus. Again in Lemmas~\ref{lem:Sgci},~\ref{lem:Bgt},~\ref{lem:Bg},
and~\ref{lem:Cg}, 
the bounding series for graphs with fixed face-width 
is of strong approximate-singular order $\alpha+1/4$
instead of $\alpha+1/2$ (which is still fine as long as the order is larger than $\alpha$), 
since there is a third case where the distinguished
non-contractible curve crosses a crosscap. 

\newcommand{\ccc}{{\hat{c}}}
\paragraph{Euler genus.}
Our results also show that, for each integer $\kappa\geq 0$ the number $e^{(\kappa)}_n$ of graphs with $n$ vertices and Euler genus $\kappa$ satisfies:
\begin{eqnarray}\label{eq:eulergenus}
e^{(\kappa)}_n \sim {\ccc}^{(\kappa)} n^{5(\kappa-2)/4-1}\gamma^n n!
\end{eqnarray}
where ${\ccc}^{(\kappa)}=\tilde{c}^{(\kappa)}$ if $\kappa$ is odd, and ${\ccc}^{(\kappa)}=c^{(\kappa/2)}+\tilde{c}^{(\kappa)}$ if $\kappa$ is even. This result follows from Lemma~\ref{lem:Cg}, from the analogous result for non-orientable surfaces,  and from the fact (shown in~\cite{RoVi90}) that for each $\kappa$ every graph of Euler genus $\kappa$ and face-width more than $2\kappa+3$ can be embedded either in the non-orientable surface $\mathbb{N}_{\kappa}$ or in the orientable surface $\mathbb{S}_{\kappa/2}$, but not both.


\section{Random graphs of  genus $g$}\label{sec:parameters}

In this section we analyze several fundamental parameters of
graphs of genus $g$ and derive limit distribution laws for them.
In all cases the limit laws do not depend on the genus, a
phenomenon that has been observed previously for \emph{maps} on
surfaces (see, for instance, \cite{GR94}). When we say that an
event holds \emph{with high probability} we mean that the
probability of the event tends to 1 as $n$ tends to infinity.
The variance of a random variable $X$ is denoted by $\sigma^2(X)$, and a sequence of random variables $(X_n)_{n\geq 1}$ is called \emph{asymptotically normal}
if $(X_n-{\bf E}[X_n])/\sigma^2(X_n)$ converges in distribution to a standard Gaussian random variable
$\mathcal{N}(0,1)$ (see~\cite[Part C]{FlaSe}).

We start with two basic parameters in order to motivate the
general analysis. We show that, as for planar graphs, the
number of edges is asymptotically normal and the number of
connected components is asymptotically Poisson distributed.

\newcommand{\E}{\hbox{\bf E}}
\newcommand{\p}{\hbox{\bf P}}

\begin{theorem}
The number of edges $X_n$ in a random graph of fixed genus $g$ with $n$
vertices is asymptotically normal and
$$
\E(X_n) \sim \kappa n , \qquad \sigma^2(X_n) \sim \lambda n
$$
where $\kappa \approx 2.21326$ and $\lambda \approx 0.43034$ are
the same constants as for planar graphs.
\end{theorem}

\begin{proof}
By Lemma \ref{lem:Gg}, the generating function $G_g(x,y) = \sum
g_{n,k} y^k {x^n/n!}$ of graphs of genus $g$ counted according to
the number of vertices and edges is of approximate-singular order
$\alpha$. This means that there exist sequences $f_{n,k}$ and
$h_{n,k}$ with
$$
    f_{n,k} \le g_{n,k} \le h_{n,k}
$$
such that  $f(x,y) = \sum f_{n,k} y^k x^n/n!$ and $h(x,y) = \sum
h_{n,k} y^k x^n/n!$ have singular expansions of order $\alpha$
with same singularity function and the same leading coefficients.
By the Quasi-powers Theorem \cite{FlaSe}, the random variables
with probability generating functions
$$
    {[x^n] f(x,y) \over [x^n] f(x,1)}, \qquad
    {[x^n] h(x,y) \over [x^n] h(x,1)}.
$$
are asymptotically normal. It follows that $X_n$, whose
distribution is given by $[x^n]G(x,y)/[x^n]G(x,1)$, also converges
to a normal law. The expectation and variance are determined
by the singularity function $\rho(y)$ as
$$
    \mu = -{\rho'(1) \over \rho(1)}, \qquad
    \sigma^2 = -{\rho''(1) \over \rho(1)} -{\rho'(1) \over
    \rho(1)} + \left({\rho'(1) \over  \rho(1)}\right)^2.
$$
By Lemma~\ref{lem:Gg}, $\rho(y)$ is independent of the genus and
we are done.
\end{proof}

Following the lines of the former proof, one can also show that
several basic parameters, such as the number of blocks, the number
of cut vertices, and the number of \emph{appearances} of a fixed
planar graph (see \cite{gimeneznoy} for a precise definition),
follow a normal law with the same moments as for planar graphs
\cite{GNR}. In order to avoid repetition we omit the corresponding
proofs. We remark in particular that, given a fixed planar graph
$H$, a random graph of genus $g$ contains a subgraph isomorphic to
$H$ with high probability; in fact, it contains a linear number of
disjoint copies of $H$.

\begin{theorem}\label{thm:components}
The number of connected components in a random graph of fixed genus $g$
is distributed asymptotically as $1+X$, where $X$ is a Poisson law
with parameter $\nu \approx 0.037439$, same as for planar graphs.
In particular, the probability that a random graph of genus $g$ is
connected is asymptotically $e^{-\nu}$.
\end{theorem}

\begin{proof}
As shown in the proof of Theorem~\ref{lem:Gg}, a graph of genus
$g$ has a unique connected component of genus $g$ with high
probability, and the remaining components are planar. Hence, in
order to study the number of components it is enough to work with
the generating function
$$
    uC_g(x) e^{uC_0(x)},
$$
the first two factors encoding the component of genus $g$, the
exponential term encoding the planar components. The generating
function of graphs with exactly $k+1$ components is then $C_g(x)
C_0(x)^k/k!$. By Lemma~\ref{lem:Cg}, the series $C_g(x)$ is
dominated coefficientwise by series having  a singular expansion
of the form
$$
    a  X^\alpha + O(X^{\alpha+1}),
$$
where $X=\sqrt{1-x/\rho}$. On the other hand, we have $C_0(x) =
C_0 + O(X)$ (see \cite{gimeneznoy}; because $\alpha<0$ we only
need the constant term in the singular expansion of $C_0$). Then
the probability of a random graph having exactly $k+1$ components
is asymptotically equal to
$$
   { [x^n]C_g(x) C_0(x)^k/k! \over [x^n] C_g(x) e^{C_0(x)}}
   \sim {aC_0^k/k! \over a e^{C_0}} = e^{-C_0} {C_0^k \over k!}.
   $$
This is precisely a Poisson distribution with parameter $C_0$, as
for planar graphs.
\end{proof}

A similar analysis as that in the previous proof shows that there
is a unique giant component of genus $g$.

\begin{theorem}\label{thm:largestC}
Let $L_n$ denote the size of the largest connected component in a
random graph of fixed genus $g$ with $n$ vertices, and let $M_n = L_n
-n$ be the number of vertices not in the largest component. Then
$$
    \p(M_n = k) \sim p \cdot g_k {\gamma^{-k} \over k!},
$$
where $p$ is the probability of a random planar graph being
connected, $g_k$ is the number of planar graphs with $k$ vertices,
and $\gamma$ is the planar growth constant.
\end{theorem}

\begin{proof}
According to the results of the previous section, the number $G_n$
of graphs of genus $g$ grows like
$$
    G_n \sim G \cdot n^{-\alpha-1} \gamma^n n!,
$$
and the number $C_n$ of connected graphs of genus $g$ grows like
$$
    C_n \sim C \cdot n^{-\alpha-1} \gamma^n n!.
$$
Using again the fact that there is a unique component of genus $g$
with high probability, we find that the probability that $M_n=k$
is asymptotically equal to
$$
\binom{n}{k}\frac{C_{n-k} g_k}{G_n},
$$
since there are ${n \choose k}$ ways of choosing the labels of the
vertices not in the largest component, $C_{n-k}$ ways of choosing
the largest component, and $g_k$ ways of choosing the complement.
Using the previous estimates we get
$$
    \p(M_n = k) \sim {C \over G} \, g_k {\gamma^{-k} \over k!}.
$$
But $C/G$ is the asymptotic probability of a graph of genus $g$
being connected, and by Theorem \ref{thm:components} it is the
same as for planar graphs.
\end{proof}

In the next two results we analyze the size of the largest block
and the size of the largest 3-connected component. For the precise
form of the Airy law of map type, a continuous distribution
defined in terms of the Airy function, and the computation of the
parameters for planar graphs, we refer to \cite{airy} and
\cite[Section 5]{GNR}.

\begin{theorem}\label{thm:largestB}
The size $X_n$ of the largest block in a random connected graph of fixed 
genus $g$ with $n$ vertices follows asymptotically an Airy law of
the map type, with the same parameters as for planar graphs. In
particular
$$
\E(X_n) \sim \alpha n,
$$
where $\alpha \approx 0.9598$, and the size of the second largest
block is $o(n^{\frac{2}{3}+\epsilon})$, for any $\epsilon>0$. Moreover the largest block has genus $g$
with high probability.
\end{theorem}

\begin{proof}
By the results of the previous section, we know that with high
probability a connected graph of genus $g$ has a unique block of
genus $g$, and the remaining blocks are planar. As we are going to
show, the unique block of genus $g$ is the largest block with high
probability.

Equation (\ref{eq:conn2conn}) encodes precisely this statement.
Since almost all 2-connected graphs have face-width $\ge 2$, it is
enough to consider the simplified composition  scheme $B_g(F(x))$,
where $F(x)=xC'(x)$ is the generating function of vertex-pointed
connected planar graphs. In the terminology of \cite{airy} this
scheme is \emph{critical} (see Lemma~\ref{lem:comp}), since the
evaluation $F(\rho)$ of $F(x)$ at its singularity is equal to the
singularity of $B_g(x)$, which is the same as the singularity of
$B_0(x)$.

By general principles (see Theorem 12 and Appendix D in
\cite{airy}) it follows that the size of the block of genus $g$
follows a continuous  Airy law.  The parameters of the law depend
only on the singular coefficients of $F(x)$, hence they are the
same as for planar graphs. In particular $\E(X_n) \sim \alpha n$,
where $\alpha = -F_0/F_2$ and $F(x) = F_0 + F_2(1-x/\rho) +
O(1-x/\rho)^{3/2}$ is the singular expansion of $F(x)$ at $\rho$.
\end{proof}

We can also adapt the proof in \cite{GNR} for the largest
3-connected component. The key point is that now the relevant
composition scheme is $T_g(x,D(x,y))$, instead of $T_0(x,D(x,y))$
as for planar graphs. Again, this is because with high probability
a 2-connected graph of genus $g$ has face-width $\ge 3$, hence it
has a unique 3-connected component of genus $g$ and the remaining
3-connected components are planar. The most technical part of the
proof in the planar case is to prove that two different
probability distributions for 2-connected planar graphs are
asymptotically equal, as the number of vertices and edges tends to
infinity at a given ratio (see \cite[Section 6.3]{GNR}). One
distribution comes from planar networks counted by the number of
edges with an appropriate weight on vertices, but since we replace
edges of the 3-connected component of genus $g$ by \emph{planar}
networks, we are dealing with the same probability distribution.
The second distribution comes from extracting the largest block in random planar
graphs with a given number of vertices. 
But we have shown in the previous theorem that the distribution of the largest block 
is asymptotically independent of the genus.

The rest of the proof (see \cite[Section 6.4]{GNR}) is easily
adapted. For computing the moments we need the asymptotic expected
number of edges in a random 3-connected graph of genus $g$, but
this is the same as for 3-connected planar graphs, since it
depends only on the singularity function of $T_g(x,z)$, which we
have proved does not depend on $g$. In conclusion, we obtain the
following result, which is analogous to Theorem 6.1 from
\cite{GNR}.

\begin{theorem}\label{thm:largestT}
The size $X_n$ of the largest 3-connected component in a random
connected graph of fixed genus $g$ with $n$ vertices follows
asymptotically an Airy law of the map type, with the same
parameters as for planar graphs. In particular
$$
\E(X_n) \sim \alpha_2 n,
$$
where $\alpha_2 \approx 0.7346$, and the size of the second largest
3-connected component is  $o(n^{\frac{2}{3}+\epsilon})$, for any $\epsilon>0$. Moreover the largest
3-connected component  has genus $g$ with high probability.
\end{theorem}

We conclude with the chromatic number of a random graph of genus
$g$. 
According to Lemma~\ref{lem:Cg}, asymptotically almost surely a random graph
of genus $g$ has face-width greater than any fixed number $k(g)$.
Taking $k(g)=2^{14g+6}$ this implies that it is 
5-colorable by a result of Thomassen~\cite{Th93}.
Moreover, as mentioned just before Theorem~\ref{thm:components}, 
a random graph of genus $g$
has a linear number of copies of $K_4$ with high probability,
in particular has chromatic number at least $4$. Consequently:

\begin{theorem}
The chromatic number of a random graph of fixed genus $g$
with $n$ vertices 
is asymptotically almost surely in $\{4,5\}$.
\end{theorem}
Unfortunately, we do not know if both values $4$ and $5$ appear on a positive proportion of graphs of genus~$g$. However, we conjecture the following:
\begin{conjecture}
The chromatic number of a random graph of fixed genus $g$
with $n$ vertices 
is asymptotically almost surely equal to $4$.
\end{conjecture}

More precise results hold for the list-chromatic number: as shown in~\cite{DevosKawaMohar}, 
a graph of fixed genus is $5$-choosable provided its face-width is large enough. Moreover, there exist planar graphs that are not $4$-choosable~\cite{voigt}. If we fix any of them, then it is asymptotically almost surely contained in a random graph of genus $g$.
Therefore we have:
\begin{theorem}
The list-chromatic number of a random graph of fixed genus $g$
with $n$ vertices is asymptotically almost surely equal to $5$.
\end{theorem}

\vspace{.2cm}

\noindent\textbf{Remark.}
All these result also hold for the random graph of fixed non-orientable genus $h$,
with exactly the same ingredients.
Since a random graph with $n$ vertices embeddable on $\mathbb{S}_g$ is asymptotically almost surely of genus $g$, and a random graph with $n$ vertices  embeddable on $\mathbb{N}_h$ is asymptotically
almost surely of non-orientable genus $h$, these
 results hold also for the random graph embeddable on $\mathbb{S}_g$
 and for the random graph embeddable on $\mathbb{N}_h$.

\bibliographystyle{plain}
\bibliography{mabiblio}

\end{document}